\documentclass[12pt]{article}
\usepackage[intlimits]{amsmath}
\usepackage{amsmath,amsxtra,amssymb,latexsym, amscd,amsthm,pb-diagram,fancyhdr}
\usepackage[mathscr]{eucal}
\usepackage{indentfirst} 
\def\leq{\leqslant}
\def\geq{\geqslant}
\def\le{\leqslant}
\def\ge{\geqslant}
\usepackage{amssymb}
\usepackage{array}
\usepackage{hhline}
\usepackage{longtable}
\usepackage{tabularx}

\usepackage[pctex32]{graphics}
\usepackage{amsmath}
\usepackage[mathscr]{eucal}
\usepackage{amsfonts}
\usepackage{amsthm}
\usepackage{longtable}
\newenvironment{myproof}[1][\proofname]{%
  \begin{proof}[#1]$ $\par\nobreak\ignorespaces
}{%
  \end{proof}
}
\newtheorem{thm}{Theorem}[section]
\newtheorem*{thm*}{Theorem}

\newtheorem{lem}[thm]{Lemma}

\theoremstyle{definition}

\theoremstyle{remark}

\def\G{{\mathscr G}}

\def\M{{\mathcal M}}

\def\Bb{{\mathbb B}}

\def\Eb{{\mathbb E}}

\def\Nb{{\mathbb N}}

\def\Pb{{\mathbb P}}
\def\Qb{{\mathbb Q}}
\def\Rb{{\mathbb R}}

\def\Zb{{\mathbb Z}}

\def\A{{\mathcal A}}

\def\D {{\mathcal D}}
\def\E {{\mathcal E}}
\def\F {{\mathcal F}}
\def\G {{\mathcal G}}

\def\J {{\mathcal J}}

\def\M {{\mathcal M}}

\def\be{\beta}
\def\de{\delta}
\def\De{\Delta}

\def\ga{\gamma}
\def\Ga{\Gamma}
\def\Om{\Omega}
\def\om{\omega}

\def\ep {\varepsilon}
\def\phi{\varphi}
\def\si{\sigma}

\def\th{\theta}

\def\et{\eta}
\def\ze{\zeta}

\def\ti{\times}

\def\bar{\overline}

\renewcommand{\proofname}{Proof}

\def\bar{\overline}

\def\leq{\leqslant}
\def\geq{\geqslant}

\def \ind {\hbox{ 1\hskip -3pt I}}

\begin{document}
\title{Monotonicity and regularity of the speed for excited 
 random walks in higher dimensions}         
\author {Cong-Dan Pham\\Aix Marseille Universit$\acute{\text{e}}$, CNRS, Centrale Marseille\\ LATP, UMR 7353, 13453 Marseille France\\cong-dan.pham@univ-amu.fr}
\maketitle
\begin{abstract}
We introduce a method for studying monotonicity of the speed of excited random walks in high dimensions, based
on a formula for the speed obtained via cut-times and Girsanov's transform. While the method gives rise to similar results as have been or can be obtained
via the expansion method of van der Hofstad and Holmes, it may be more
palatable to a general probabilistic audience. We also revisit the law of large numbers for
stationary cookie environments. In particular, we introduce a new notion of
$e_1-$exchangeable cookie environment and prove the law of large numbers for
this case.
\end{abstract}
\section{Introduction}

\subsection{Excited random walk with random cookies (ERWRC)} 

Excited random walks (ERW) were introduced in \cite{BW03} by I. Benjamini and D. Wilson. After that, M. Zerner generalized ERW when he introduced in \cite{Zer05}, \cite{Zer06} cookie random walks, which are also called multi-excited random walks. In our paper, we consider a model of random walk called the excited random walk with $m$ random cookies  which we denote by ERWRC or  $m$-ERWRC when more explicitly needed. This is a generalisation of multi-excited random walk and  also a particular case of the excited random walk in random environment introduced and considered in \cite{MPRV12} and \cite{KZ14}.

Let us describe $m-$ERWRC.  Let $m$ be a positive integer or $m=+\infty$. We place $m$ cookies on 
every site of the lattice $\Zb^{d}.$ Moreover,  $m$ random variables 
$(\be_k(y))_{1\leq k\leq m}$ with values in $[-1,1]$ are attached to each site $y$ of $\Zb^{d}$. The process
$ \be:=\{(\be_k(y))_{1\leq k\leq m}\}_{y\in\Zb^{d}}$ serves as a random environment whose law 
is denoted by $\Qb$. Let $\Bb:=([-1,1]^{m})^{\Zb^{d}}$ be the set of random environments. The excited random walk with $m$ cookies 
$ \be=\{(\be_k(y))_{1\leq k\leq m}\}_{y\in\Zb^{d}}$ 
is a discrete time nearest neighbor random walk $(Y_n)_{n\ge 0}$ on the lattice $\Zb^d$ obeying the following rule: when the walk visits $y$ for the $k$-th time, $1 \le k \le m$, then it eats one cookie 
and jumps with probability $(1+\be_{k}(y))/2d$ to the right, probability $(1-\be_{k}(y))/2d$ to the left, 
and probability $1/(2d)$ to the other nearest neighbor sites. On the other hand, when the walk is at a site $y$ 
where there is no more cookie, then it jumps uniformly at random with probability $1/(2d)$ to one of the $2d$ neighboring sites.
When $m=1$ and the environment $\be$ is constant, we recover the excited random walk.

Throughout this paper, we denote by $\{Y_n\notin^k\}$ the event that $Y_n$ has been visited fewer than $k$ times before time $n$ and denote by $\{Y_n\in^k\}$ the complement  of $\{Y_n\notin^k\}.$ When $k=1$ we also use the notations $\{Y_n\notin\}:=\{Y_n\notin^1\}$ and $\{Y_n\in\}:=\{Y_n\in^1\}.$ Moreover, the event that $Y_n$ has been exactly visited $k-1$ times before time $n$ is denoted by $\{Y_n\notin_k\}$ and its complement is denoted by $\{Y_n\in_k\}$.

From the description of $m-$ERWRC, when $\be$ is fixed, the ``quenched" law $\Pb_{\be}$ of excited random walk with $m$ random cookies $ \be$ is the probability on the path space $(\Zb^d)^\Nb$, defined by: 
\begin{itemize}
\item $\Pb_{\be}(Y_0=0)=1$,
\item $\Pb_{\be}[Y_{n+1}-Y_n=\pm e_i|Y_0,..., Y_n]=\frac{1}{2d}$ for $2\leq i\leq d$,
\item if $Y_n$ has been visited exactly $k-1$ times before time $n$, i.e. on the event $\{Y_n\notin_k\}$
$$
\Pb_{\be}[Y_{n+1}-Y_n=\pm e_1|Y_0,..., Y_n]=
\left\{ \begin{array}{ll}
\frac{1\pm\be_{k}(Y_n)}{2d} & \mbox{ for } 1 \le k \le m, \\
 \frac{1}{2d} &  \mbox{ for } k > m.
\end{array}
\right.
$$
\end{itemize}
The ``annealed" law $P$ is  then defined as the semi-direct product on $\Bb\times(\Zb^d)^\Nb$: $P=\Qb\otimes\Pb_{\be}.$
We say that the cookies are ``identical" if 
\begin{gather*}\tag{IDEN}\label{IDEN}
\forall k \mbox{ such that } 1\leq k\leq m \, , \, \, 
\forall y \in \Zb^{d} \, , \, \, \be_k(y) = \be(y) \,.
\end{gather*} 

In this model, the random cookie environment $\be=\{\be(y)\}_{y\in\Zb^{d}}$ is assumed to be:
\begin{itemize}
\item  stationary: $\be(y+\cdot)\overset{law}{=}\be$ for any $y$ in ${\Zb^{d}}$,
\item  $e_1$-exchangeable: to define this notion,  we consider a family  
$\De=\{\de_z\}_{z\in\Zb^{d-1}}$ of bijective mappings from $\Zb$ to $\Zb$. The mapping
$\si_{\De}: \Zb^d \to \Zb^d$ defined by $\si_{\De}(x,z)=(\de_z(x),z)$ for all $x\in\Zb, z\in\Zb^{d-1}$, is then 
a bijection from $\Zb^d$ to $\Zb^d$, acting on the set $\Bb$ of environments by 
$\si_{\De}(\be)(y)=\be(\si_{\De}(y))$. The environment is said to be $e_1$-exchangeable if and only if
$\si_{\De}(\be)  \overset{law}{=}\be$ for any family $\De$. In other words, an environment is $e_1$-exchangeable
if its law does not change when performing permutations of the environment on each
horizontal line. 
\end{itemize}

An i.i.d. cookie environment is of course stationary and $e_1$-exchangeable. Another simple example is provided by
a stationary environment not depending on the horizontal component: for all $y=(x,z) \in \Zb \times \Zb^{d-1}$,
$\beta(y)=\beta(z)$, where $(\beta(z))_{z \in \Zb^{d-1}}$ is stationary. 

To describe our main result about this model, we introduce a partial ordering on the laws of environments. Generally 
speaking, let $Q_1, Q_{2}$ be two probability measures on a partially ordered set $(E,\leq)$. We say that 
a  probability measure $Q$ on $E \times E$  is a monotone coupling
of  $Q_1$ and  $Q_2$, if when denoting by  $l_1$ and $l_2$  the coordinate 
maps from $E \times E$ to $E$:
$$ \mbox{ for } i=1,2 \, ,  \, \text{ for all } B \mbox{ events of } E \, , \, \, Q(l_i \in B)=Q_i(B) \, 
\mbox{ and } Q(l_1 \leq l_2) =1 \, .
$$
When such a monotone coupling exist, we say that $Q_1 \prec Q_2$. 
  
The set $\Bb$ of environment is provided with the partial ordering:
$$ \beta_1 \leq \be_2 \mbox{ if and only if } \beta_{1,k}(y) \leq \beta_{2,k}(y) \, , \, \, 1 \leq k \leq m
	\, , \, \,  y \in \Zb^d \, .
$$
Let $(Z_n)_{n\geq 0}$ (resp. $(X_n)_{n\geq 0}$) be the vertical (resp. horizontal) component of $m-$ERWRC $(Y_n)_{n\geq 0}$:
 $$Z_n:=(Y_n\cdot e_2,...,Y_n\cdot e_d) \, , \, \,  X_n:=Y_n\cdot e_1 \, . 
 $$
Then $(Z_n)_{n \ge 0}$ is a simple random walk on $\Zb^{d-1}$. We can extend this simple random walk to times integer to obtain the simple random walk $(Z_n)_{n \in\Zb}$ (see \eqref{lienZ-tildeZ} in Section \ref{contructES}). For $d-1\geq 5$,  E. Bolthausen, A-S. Sznitman and O. Zeitouni \cite{BSZ03}  proved the existence of cut times, i.e. times splitting 
the trajectory into two non-intersecting paths. Moreover, these cut times are integrable for $d-1\geq 5$. Let $\D$ be the 
set of cut times, write $\D=\{...<T_{-2}<T_{-1}<T_0\leq 0<T_1<T_2<...\}$. We denote $T:=T_1$ and $\hat{P}=P(\cdot|0\in\D)$.
 Our main result reads then  as follows:
\begin{thm} \label{Therwrc} Let $Y=(Y_n)$ be $m-$ERWRC, assume that the random cookie environment is stationary and $e_1$-exchangeable. We denote $X_n=Y_n\cdot e_1$ the projection of the random walk on the first coordinate. 
\begin{itemize}
 \item Law of large numbers:
 
 For $d\geq 6$,   $\frac{X_n}{n}$ converges $P-$a.s. to a random variable $V$, 
whose expectation under $\hat{P}$ is denoted by $v(\Qb)$ satisfying $v(\Qb)=\hat{E}[V]=\frac{\hat{E}(X_T)}{\hat{E}(T)}$. In the particular case that the cookie environment is $i.i.d.$ then $V$ is constant and $V=v(\Qb).$
\item Monotonicity: 

\begin{enumerate}
\item If the cookies are identical $i.e.$ $\forall y \in \Zb^d \, , \, \, \be_1(y)=\be_2(y)=...=\be_m(y)=\be(y) \,  
$ then there exists $d_0\in\Nb^{*}$  such that 
$v(\Qb)$  is increasing w.r.t. $\Qb$ for $d\geq d_0$ (w.r.t. the partial ordering $\prec$).
\item If the cookies are identical, there exists $\si_0\in(0,1)$ such that for any $d \geq 10$, 
$v(\Qb)$  is increasing w.r.t. $\Qb$ on the set 
$\{\Qb \mbox{ such that }   \Qb(0 \le |\be(y)| \le \si_0, \forall y \in \Zb^d)=1 \}$.
\end{enumerate}
\end{itemize}

%
%
\end{thm}

About the law of large numbers (LLN):

To prove the law of large numbers, we use the technique of cut times as in 
\cite{BSZ03}, \cite{HoSu12}, \cite{Hol12}. Our contribution is to use it
for $e_1-$exchangeable stationary environment. In the i.i.d. setting, the LLN of Theorem $\ref{Therwrc}$ is a consequence of Theorem $1.1 $ of \cite{HoSu12}. However, the proof of the LLN for i.i.d. setting in our paper is not totally the same as in \cite{BSZ03}, \cite{HoSu12} (see Section \ref{seciid}). The formula $v(\Qb)=\frac{\hat{E}(X_T)}{\hat{E}(T)}$ obtained in our proof is different from the formula of the speed in $\cite{BSZ03}$. To use cut times, the dimension $d$ is required to be not smaller than $6$. This implies the existence of cut times of the projection $Z$ of the random walk $Y$ on the $d_1$ last coordinates ($d_1=d-1\geq 5$). In \cite{KZ14}, the LLN of Theorems $4.6$ and $4.8$ is proved for all dimension $d\geq 1$ using renewal structure. However, to use this technique, the conditions of uniform ellipticity of the cookie environment and the transience of the random walk in some direction $l\in \Rb^d $ are needed.

About the monotonicity of the speed:

Our result is to prove for the case of $e_1-$exchangeable and stationary cookie environment. For the i.i.d. setting,  M. Holmes and R. Sun \cite{HoSu12} considered random walks in partially random
environment which is similar to random walks with an infinite number of
identical random cookies, $m=+\infty$. In this model the probability of stepping in $d_1$ last coordinates is random (i.e. this probability depends on the cookie environment) and $d_0=d-d_1\geq 1$. The question of monotonicity is considered under the assumption that there is an explicit coupling of two laws of the random environments. In this case, the laws of the random environments are allowed to take two values, say $\nu_1$ and $\nu_2$ with probabilities
$\be$ and $1-\be$ (where $\be$ is a constant in $[0,1]$). They proved the monotonicity of the speed with respect to $\be.$ 

In the paper, we prove the monotonicity for all $1\leq m \leq +\infty$. In fact, there is an intersection between our model and the model used in \cite{HoSu12} that the projected random walk on $\Zb^{d_1}$ is a simple random walk, $m=+\infty$ and $d_0=1$. In this case, the monotonicity can be easily proved by coupling argument for stochastic domination (see \cite{HoSu12}, page $5$). 

 About the methodology, for the case of the probability of stepping in $d_1$ last coordinates is not random, with the method cut times and Girsanov's transform, the explicit coupling of the laws of random environments used in \cite{HoSu12} is not needed in our proof. We notice that the lace expansion method can be applied to prove the monotonicity of the speed of Theorem \ref{Therwrc}. More precisely, using the stationary coupling $\be_t=(1-t)\be_1+t\be_2$, we can prove  the existence of the speed by the law of large numbers. Together with boundedness and convergence of the lace expansion series, the lace expansion formula for the expectation of the speed then follows. These techniques can be found in  \cite{HoSu12}.

In \cite{Hol12}, M. Holmes asked about monotonicity of the speed with
respect to stochastic domination. He considered the model with $1\leq m\leq +\infty$, $d_0=1$ and  the probability of stepping in $d_1$ last coordinates is not random. The author proved the following result:

\begin{thm*}[Theorem 2.3, \cite{Hol12}]\label{ThHol} Set $\de_i:=\Eb[\be_i(0)]$. Let $A$ be a finite set of integers $A\subset N$. If $\be_i(o)$ is independent
of $(\be_j(o))_{j\ne i}$ for each $i\in A$, then for each fixed joint distribution of $\be_{A^c}(o) = (\be_i(o))_{i\notin A}$,
the annealed speed $v$ in dimension $d$ is a continuous function of $(\de_i)_{i\in A}$ when $d\geq 6$ and is differentiable in $\de_i$ for each $i\in A$ when $d\geq 8$. If $1\in A$, then $v$ is strictly increasing in $\de_1$ when $d\geq 12$.
\end{thm*}
Under the conditions of this theorem, for $i\in A\text{ and }i>0$, the speed depends on the $\be_i$ via the mean $\de_i=\Eb[\be_i(x)]$ where $x\in\Zb^d$. This means that the law of the random walk does not change when we replace $\be_i, i\in A$ by the constant $\de_i$. Here the speed is monotone in the first drift $\de_1$ when the $i$-th cookie is independent of the others for $i\in A$ and $1\in A$. This is a special case of stochastic domination. The model in our paper is quite similar to the model in \cite{Hol12} except the conditions of the random cookie environment. We prove the monotonicity of the speed with respect to the law of the random cookie environment $\Qb$ for the special case of $m$ identical random cookies.


\subsection{Excited random walk with $m$ identical deterministic cookies ($m$-ERW)}
This model is a partial model of $m-$ERWRC when the cookie environment is not random and identical, i.e. the cookies are the same for every site:
$$ \forall k \mbox{ such that } 1\leq k\leq m \, , \, \, 
\forall y \in \Zb^{d} \, , \, \, \be_k(y) = \be \, , 
$$
for some real number $\be\in[0,1]$. We see that the $m-$ERW is also a partial model of the model called multi-excited random walk which was introduced in \cite{Zer05}. Let $\Pb_{m,\be}$ denote the law of  $m$-ERW. As $m$ is large,
the $m$-ERW is more and more like a simple random walk with bias $\be$. Let $v(m,\be)$ be 
the speed of the $m$-ERW  whose existence is proven for $d \geq 2$ in \cite{BR07}, \cite{MPRV12}, \cite{KZ14}.
 We prove in Section \ref{Sec4} the following result:  
\begin{thm}\label{md}
For $d\geq 8$, the speed  $v(m,\be)$ is differentiable w.r.t $\be$ in $[0,1)$.  Moreover, the derivative 
converges to
$\frac{1}{d}$, uniformly in $\be$ on compact subsets of $[0,1)$: for any $\be_0 \in [0,1)$,
$$ \lim_{m\to\infty} \sup_{\be \in [0,\beta_0]} \left| \frac{\partial}{\partial\be}v(m,\be) - \frac{1}{d} \right| =0 \, .$$
Hence, there exists $m(\be_0)$ such that for $m\geq m(\be_0)$ 
the speed of the $m$-ERW is increasing in $\be$ on $[0,\be_0]$.

\end{thm}
The differentiability of the speed was proved in \cite{Hol12} Theorem 2.3. The rest could also be obtained by minor modification of the proof of Theorem $2.3$ of \cite{Hol12}.
\subsection{Excited random walk}

Excited random walk is introduced in \cite{BW03}, this model is a partial case of $m-$ERW when $m=1$.
Our main result  for the excited random walk is the following:

\begin{thm}Let $v(\be)$ be the speed of ERW with bias $\be.$
\label{ThERW}
\begin{enumerate} 
\item $v(\be)$ is differentiable in $\be\in [0,1)$ for $d\ge 8$. For $d\geq 6$,
 the derivative at the critical point $0$ exists, is positive  
 and satisfies :
 $$\lim_{\be\to 0}\frac{v(\be)}{\be}=\frac{1}{d}R(0) \, , $$
where $R(0):=\lim_{n\to\infty}(R_n/n)$, $R_n$ is the number of points visited  at time $n$ by the  symmetric simple random walk on $\Zb^{d}.$
\item There exist $d_0\in \mathbb{N}^{*},\,\be_0\in(0,1)$ such that the speed of the excited random walk is strictly increasing in $\be \in\left[0,1\right]$ for $d\geq d_0$ and strictly increasing in $\be\in[0,\be_0)$ for $d\ge 8$.
\end{enumerate}
\end{thm}
For the monotonicity of the speed in a neighborhood of $0$, we need $d\geq 10$ in Theorem \ref{Therwrc}, but in Theorem \ref{ThERW} here, we need only $d\geq 8.$
In the point $1$ of Theorem \ref{ThERW}, the differentiability of $v(\be)$ on $[0,1)$ for $d\geq 8$ is contained in Theorem $2.3$ of \cite{HoSu12}. However, we add the differentiability at the critical point $0$ for $d\geq 6$. The point $2$ is proved in \cite{vdHH10} for $d_0=9$ by the lace expansion method.

In our paper, we prove the results by using cut times and Girsanov's transform. 
Our proof is based on two ingredients:
\begin{itemize}
\item Using stationary properties, it is possible to express the expectation of the speed in the direction $e_1$ as follows:
\begin{equation}\label{F1}
v(\Qb)=\frac{\Qb\Eb_{\be}(X_T|0\in\D)}{E_{\be}(T|0\in \D)},
\end{equation}
where $\Eb_{\be}$ is the expectation under  the ``quenched'' law $\Pb_{\be}$ of ERWRC, $\D$ is the 
set of cut times, and $X_T=Y_T\cdot e_1$. In the case of i.i.d. cookies, $v(\Qb)$ is also the speed when the speed is deterministic.
\item Starting from $\eqref{F1}$, we consider two random cookies $\be_1$ and $\be_2$, and a stationary coupling $\be_t=(1-t)\be_1+t\be_2, t\in[0,1].$ We get the expectation of the speed for random cookies $\be_t$ (see \eqref{vitesst}) as follows:
\begin{align} \label{F2}
  f(t):=\frac{\Qb\Eb_{\be_t}(X_T \, 1_{0\in\D})}{E(T \, 1_{0\in\D})}.
\end{align}
Where, $E$ is the expectation w.r.t. $P$. We use Girsanov's transforms to make the dependence of $f(t)$ w.r.t to 
$t$ more explicit. This enables us to compute the derivative of $f(t)$ when $d\geq 8$ and to prove that this derivative is positive for $d$ high enough or if random cookie is small enough to $0$. 
\end{itemize}

All results of differentiability and monotonicity in \cite{vdHH10}, \cite{Hol12},
\cite{HoSu12} were proved by using the lace expansion. We do not use this method in this paper.  In \cite{HoSu12}, the authors used cut times to prove the law of large numbers. The existence of the speed and the convergence of the lace expansion series allow to express the speed by the lace expansion formula. This formula was used in calculating the derivative and showing that the derivative is positive. 
In this paper, to prove the law of large number we also use the cut times. However, with different arguments, we obtain the formulas of the speed (see \eqref{F1} and \eqref{F2}), which are more explicit than the formulas in the previous works. 

In other to prove the monotonicity of the speed, we do not use the lace expansion formula, we use directly the formulas \eqref{F1} and \eqref{F2} of the speed via cut time $T$. These formulas have the
advantage that the denominator $\Eb_{\be}(T|0\in \D)$ does not depend on random cookie $\be$. Girsanov's transform gives an expression of the derivative $\frac{\partial f}{\partial t}(t)$ via the cut time $T$ (see \eqref{dht}). Using this formula, we estimate the derivative of the speed and obtain that the derivative is positive when $d$ large enough depending on the moments of $T$. Here, the condition $d\geq 6$ is needed for the existence of cut time, and we have $\sup_{d\geq 6}\hat{E}T=\sup_{d\geq 6}\frac{1}{P(0\in\D)}<+\infty$. 

In the proof of the monotonicity in Theorem \ref{Therwrc}, in the estimation of the derivative, there is the appearance of the third moment of cut time $T$ (see \eqref{T3}). Therefore, we need $d\geq 10$ to get $\sup_{d\geq 10}\hat{E}(T^3)<+\infty$. For the particular case of ERW, the second moment of cut time $T$ appears (see \eqref{T2}). Hence, we need $d\geq 8$ to have that $\sup_{d\geq 8}\hat{E}(T^2)<+\infty.$ 

Notice that the constant $d_0$ in our method depends on the moments of $T$. While by using lace expansion method, M. Holmes and co-authors gave a explicit integer $d_0$, example in \cite{vdHH10} $d_0=9$, in Theorem $2.3$ of \cite{Hol12} $d_0=12$.


The paper is organized as follows: in Section \ref{Sec2}, we prove Theorem \ref{Therwrc}. First we give a construction of $m-$ERWRC. We then prove the law of large numbers and obtain an expression of the
speed by using cut times for stationary and $e_1-$exchangeable cookies. In the particular case of i.i.d. cookies, we prove that the speed is deterministic. Using Girsanov's transforms, we get the
derivative of the speed and estimate it to obtain the differentiability
and monotonicity of the speed. Section \ref{Sec3} is devoted to the proof of Theorem \ref{ThERW} based on that of Theorem \ref{Therwrc}. In Section \ref{Sec4}, we prove Theorem \ref{md}.  The key of the proof is Lemma \ref{bd2}. We use this lemma to show that the derivative of the speed tends uniformly in the drift $\be$ to a positive constant when the number of cookies tends to the infinity.

\section{Proof of Theorem \ref{Therwrc}}\label{Sec2}
\subsection{A construction of $m-$ERWRC} \label{contructES}
We begin this section by constructing the $m-$ERWRC from some independent 
sequences of random variables. This plays an important role to prove the monotonicity. 
Fix $\be(y)=(\be_1(y),\be_2(y),...,\be_m(y)), y\in\Zb^d.$ First, we consider a simple random walk (SRW) $\{ \tilde{Z}_n \}_{n\in\Zb}$ on $ \Zb^{d-1}$ where $\tilde{Z}_0:=0$. Let three sequences of random variables and random vectors
$\{\et_i\}_{i\geq 0}$,$\{\xi_i\}_{i\geq 0}$ and $\{\ze_1(y),...,\ze_m(y)\}_{y\in\Zb^d}$ such that every random variable in these sequences is independent of each other,
 independent of $\tilde Z$ and having distribution
$$\et_i\sim Ber\left(\frac{1}{d}\right),\quad \xi_i\sim Ber\left(\frac{1}{2}\right),
\quad \ze_k(y)\sim Ber\left((\be_k(y)+1)/2\right)\text{ where } 1\leq k\leq m.$$
$\{\tilde{Z}_n\}_{n\geq 0}$ will give the sequence of vertical moves of the excited random walk, 
${\et_i=+1}$ will mean that at time $i$, the excited random walk performs an horizontal move. The direction 
of this move is given by $\xi_i$ when the $m-$ERWRC is at a site that has been visited more than $m-1$ times before the time $i$ , and by $\ze_k(y), k\in\{1,2,...,m\}, y\in\Zb^d$ otherwise. 
More precisely, set $A_i^n:=\{\sum_{j=0}^{n-1}(1-\et_j)=i\}$, $(0\leq i\leq n)$ for $n>0$ and $A_0^0:=\Om$. 
Then for every $n\geq 0$, we have $\bigcup_{i=0}^{n}A_i ^n=\Om$ and $A_i ^n\bigcap A_j^n=\emptyset$ for $i\neq j$. 
We define the vertical component $Z$ of $Y$ by: 
\begin{equation}
\label{lienZ-tildeZ}\forall n\in\Zb , Z_n= \begin{cases}\tilde{Z}_{0} & \mbox{ if } n = 0 \, , \\
		\tilde{Z}_{\sum_{i=0}^{n-1}(1-\et_i)} & \mbox{ if } n > 0 \, , \\
		\tilde{Z}_{-\sum_{i=n}^{-1}(1-\et_i)} & \mbox{ if } n < 0 \, .
		\end{cases}
\end{equation}
We  now construct the horizontal component $X$ of $Y$. 
Set $Y_0:=0$ and  assume that $(Y_j,  0 \leq j\leq i)$ are constructed. Let us define $Y_{i+1}$. 
On the event $Y_i\notin_k$ i.e. $Y_i$ has been exactly visited $k-1$ times before time $i$, set
$$ 
\E_i:=\begin{cases} (2 \ze_k(Y_i) -1) \ind_{\et_i=1}&\text{ if }1\leq k\leq m\, ,\\
(2 \xi_i -1) \ind_{\et_i=1}&\text{ if } k>m\,.
\end{cases}
$$
We then set $X_{i+1}:=X_i + \E_i$, and $Y_{i+1}:=(X_{i+1},Z_{i+1})$.
With this construction, we obtain:
\begin{lem}\label{3}$Y$  is a $m-$ERWRC of the quenched law $\Pb_{\be}$.
\end{lem}
\begin{proof}
For the proof of Lemma $\ref{3}$, we need the following lemma:
\begin{lem} \label{4}Let $\F$ and $\G$ be two sigma-algebras and $C\in \F\cap\G$ 
such that $\F|_C:=\{A\cap C \mbox{ with }A\in\F\}\subset \G$. For any integrable random variable $V$, we 
get 
$$\Eb(V1_C|\F)=\Eb\left[\Eb(V1_C|\G)|\F\right].$$
\end{lem}
The proof of Lemma $\ref{4}$ is easy.
Now, we return to the proof of Lemma $\ref{3}$.
Set
\begin{align*}
&\F_n^{Y}:=\si(Y_j, 0 \leq j\leq n)\\
&\F_n:=\si(Z_j, 0 \leq j\leq n,\et_j,\xi_j,\ze_k(y),\, 0 \leq j\leq n-1, 1\leq k\leq m, y\in\Zb^d)\\
&\G_{ni}:=\si (\tilde Z_j, 0 \leq j\leq i,\xi_j,\et_j,\ze_k(y),\, 0 \leq j\leq n-1, 1\leq k\leq m, y\in\Zb^d)
\end{align*}
It is clear that $\F_n^{Y}\subset \F_n \mbox{ and } A_i^n\in \F_n\cap \G_{ni}$. 
Moreover, $\F_n|_{A_i^n}\subset \G_{ni}$. 
Now, using Lemma $\ref{4}$, we have for $j\geq 2$,
\allowdisplaybreaks{
\begin{align*}
&\Pb(Y_{n+1}-Y_n=\pm e_j|\F_n^{Y})\\
&=\sum_{i=0}^{n}\Pb\left(\tilde Z_{i+1}-\tilde Z_i=\pm e_j,A_i^n,\et_n=0|\F_n^{Y}\right)\\
&=\sum_{i=0}^{n}\Pb\left[\Pb\left(\tilde Z_{i+1}-\tilde Z_i=\pm e_j,A_i^n,\et_n=0|\F_n\right)|\F_n^{Y}\right]\\
&=\sum_{i=0}^{n}\Pb\left[\Pb\left(\tilde Z_{i+1}-\tilde Z_i=\pm e_j,A_i^n,\et_n=0|\G_{ni}\right)|\F_n^{Y}\right]\\
&=\sum_{i=0}^{n}\Pb(\tilde Z_{i+1}-\tilde Z_i=\pm e_j)\Pb(\et_n=0)\Pb(A_i^n|\F_n^{Y})\\
&=\Pb(\tilde Z_{i+1}-\tilde Z_i=\pm e_j)\Pb(\et_n=0)=\frac{1}{2(d-1)}.\left(1-\frac{1}{d}\right)=\frac{1}{2d}.
\end{align*}}
For the case $e_j=e_1$, on the event $Y_n\notin_k$ where $k\leq m$,
\begin{align*}
&\Pb(Y_{n+1}-Y_n=+ e_1|\F_n^{Y})=\Pb(\et_n=1,\E_n=1|\F_n^{Y})\\
=&\Pb(\et_n=1,\ze_k(Y_n)=1|\F_n^{Y})=\Pb(\et_n=1).\Pb(\ze_k(Y_n)=1|\F_n^{Y})\\
=&\frac{1}{d}.\frac{1+\be_k(Y_n)}{2}=\frac{1+\be_k(Y_n)}{2d}.
\end{align*}

The cases $e_j=-e_1$ and $k>m$ are treated similarly. Lemma $\ref{3}$ is now proved.
\end{proof}

Now, set $\D:=\{n\in \Zb \mbox{ such that }Z_{(-\infty,n)}\cap Z_{[n,+\infty)}=\emptyset\}$ to be 
the set of cut times of $Z$ and similarly let $\tilde\D$ be the set of cut times of $\tilde{Z}$. 
The sequence of cut times of $Z$ is then defined by induction:
\begin{align*}
T_1&:=\inf\{n>0\mbox{ such that } n\in\D\},\\
T_{i+1}&:=\inf\{n>T_i\mbox{ such that } n\in\D\}\,\, \text{, for } i\geq 1,\\
T_{i-1}&:=\sup\{n<T_i\mbox{ such that } n\in\D\} \text{, for } i\leq 1. 
\end{align*}
By construction, $T_0\le 0< T_1$ and we set $T:=T_1.$ We define similarly $\tilde{T}_i$ 
and $\tilde{T}$ for $\tilde{Z}$. 
Observe that the laws of $T$ and $\tilde{T}$ do not depend on the environment $\be$, since 
 they depend only on $Z$ and $\tilde{Z}$.  
Moreover, it follows from (\ref{lienZ-tildeZ}) that 
\begin{equation}\label{T}
\tilde{T}=\sum_{i=0}^{T-1}(1-\et_i) \, , \, \, \text{ and } 
\{T>k\}=\left\{\tilde{T}>\sum_{i=0}^{k-1}(1-\et_i)\right\} \, .
\end{equation}

We consider $W:=\{\om\in\Om:\forall j\,,T_j(\om)<\infty\}$. 
E. Bolthausen, A-S. Sznitman and O. Zeitouni \cite{BSZ03} proved that $\Pb(W)=1$ and $\Pb(0\in\D)>0$ 
for $d-1\geq 5.$
Let $\hat{\Pb}:=\Pb(.|0\in\D)$ be the Palm measure.

Exactly in the same way, we can prove (see  Lemma 1.1 of \cite{BSZ03})
  the following lemma:
\begin{lem}\label{EsT}
Let $f$ be a non-negative measurable function, for $d\geq 6$ we have 
\begin{equation}
\label{P-hatP}
\int fd\Pb=\frac{\int \sum_{k=0}^{T-1} f\circ\th_k \, d\hat{\Pb}}{\int Td\hat{\Pb}}
\, 
\end{equation}
with convention that one of two sides equals to $+\infty$ so the other equals to $+\infty$.
A simple instance of this formula is to take $f=\ind_{0 \in \D}$, so that $\sum_{k=0}^{T-1} f\circ\th_k =1$,
leading to 
\begin{equation}
\label{Ecutime}\Pb(0 \in \D) = (\hat{\Eb}{T})^{-1} \text{ and }\Eb[T1_{0\in\D}]=1.
\end{equation} 
\end{lem}
\begin{proof}
Indeed, by Lemma 1.1 of \cite{BSZ03}, \eqref{P-hatP} is true for $f.1_{f\leq c}$ for some positive constant $c$. Take $c$ to tend to $+\infty$ we get \eqref{P-hatP}.
\end{proof}

It is proved in \cite{BSZ03} that $\hat{\Eb}T=1/\Pb(0\in\D)<\infty$ for $d\geq 6$, $\Eb T<+\infty$ 
when $d\geq 8$ and $\Eb (T^2)<+\infty$ 
when $d\geq 10$. 
Hence we can take $f=T$ in \eqref{P-hatP}. Observe that $T\circ\th_k=T-k$ for $k\in\{0,1,2,...,T-1\},$ 
\eqref{P-hatP}  reads
\begin{align}\label{7}
\hat{\Eb}T \, \Eb T=\int \left[T+(T-1)...+1)\right]d\hat{\Pb}
= \hat{\Eb} \left( \frac{T^2 + T}{2} \right) 
\, .
\end{align}
Now, we take $f=T^2$, observe that $T^2\circ\th_k=(T-k)^2$ for $k\in\{0,1,2,...,T-1\},$ 
\eqref{P-hatP}  reads
\begin{align}\label{ET3}
\hat{\Eb}T \, \Eb (T^2)=\int \left[T^2+(T-1)^2...+1^2)\right]d\hat{\Pb}
= \hat{\Eb} \left[ \frac{T(T+1)(2T+1)}{6} \right] 
\, .
\end{align}
Therefore,   
\begin{equation}
\hat{\Eb}(T^2)<+\infty \text{ for } d\geq 8,\,\hat{\Eb}(T^3)<+\infty \text{ for } d\geq 10.
\end{equation}
Actually,  Lemma \ref{TL2} asserts the stronger result that 
$$c_1:=\sup_{d\geq 8}\hat{\Eb}(T^2)<+\infty.$$
To prove monotonicity of the speed, we need the moments of $T$ are bounded as in the following lemma:
\begin{lem}
\label{TL2}
$$c_1:=\sup_{d\geq 8}\hat{\Eb}\left(T^2\right)<+\infty$$ and 
$$
c_2:=\sup_{d\geq 10}\hat{\Eb}{(T^{3})}<+\infty
$$
\end{lem}

\begin{proof}
From $\eqref{7}$, we have $$\hat{\Eb}(T^2)
=2\Eb T\hat{\Eb}T-\hat{\Eb}T=\frac{2\Eb T-1}{\Pb(0\in\D)}.$$
Because $\lim_{d\to+\infty}\Pb(0\in\tilde{\D})=1$ and $\Pb(0\in\D)=\frac{d-1}{d}\Pb(0\in\tilde{\D})$
 (see \cite{ET60}, remark 3, page 248), to show that $c_1<+\infty$ (resp. $c_2<+\infty$), it is enough to prove that $\sup_{d\geq 8}\Eb (T)<+\infty$ (resp. $\sup_{d\geq 10}\Eb (T^2)<+\infty$). \\
Choose $\ep$ such that $0<\ep<1$. We consider a simple random walk $Z^{\ep}$ on $\Zb^{d-1}$such that:
\begin{align}
&\Pb\left(Z^{\ep}_{n+1}-Z^{\ep}_n=e|\F^{Z^{\ep}}_n\right)=\frac{\ep}{2(d-1)}, \text{ for } e\in\{\pm e_2,\pm e_2,...,\pm e_{d}\}\notag,\\
&\Pb\left(Z^{\ep}_{n+1}-Z^{\ep}_n=0|\F^{Z^{\ep}}_n\right)=1-\ep.
\end{align}
Note that, we can construct $Z^{\ep}$ from the sequences $(\tilde{Z}_n)_{ n\in\Zb}$, $(\et^{\ep}_n)_{ n\in\Zb}$,
 where $\et^{\ep}_{n}\sim Ber(1-\ep)$ as in the construction of $Z.$
Set $\J:=\{n \text{ such that }Z^{\ep}_n\ne Z^{\ep}_{n-1}\}$ and write $\J=\{...<j_{-1}<j_0\leq 0<j_1<...\}.$ Set $\mu_n:=j_n-j_{n-1}\text{ for }n>1$ and $\mu_1:=j_1.$ Then, the $(\mu_n)_{n\geq 0}$ are i.i.d. , Geometric$(\ep)$ random variables.  We call $\{T^{\ep}_n\}_{n\in\Zb}$ the cut times of $Z^{\ep}$, $T^{\ep}:=T^{\ep}_1$ 
 and $\D^{\ep}$ is the set of cut times.
Then $\Pb(0\in\D^{\ep})=\ep\Pb(0\in\tilde{\D})$ converges to $\ep$ when $d\to\infty$ and $\Pb(0\in\D^{\ep})$ is bounded  by $\ep.$ We also have $T^{\ep}=\sum_{i=1}^{\tilde{T}}\mu_i$. Then
\begin{align}
\Eb(T^{\ep})&=\sum_{k\geq 1}\Eb(\sum_{i=1}^{k}\mu_i)\Pb[\tilde{T}=k]\notag\\
&=\sum_{k\geq 1}\frac{k}{\ep}\Pb[\tilde{T}=k]\notag\\
&=\frac{\Eb \tilde{T}}{\ep}.
\end{align}
We compute similarly and get that 
\begin{align*}
\Eb[(T^{\ep})^2]=\frac{\Eb(\tilde{T}^2)+(1-\ep)\Eb(\tilde{T})}{\ep^2}.
\end{align*}
$T$ is $T^\ep$ with $\ep=\frac{d-1}{d}$ then $\Eb T=\frac{d}{d-1}\Eb\tilde{T}$, so that 
$\Eb T=\frac{d\ep}{d-1}\Eb(T^{\ep})$. Therefore, in order to prove that $\sup_{d\geq 8}\Eb T<+\infty$ (resp. $\sup_{d\geq 10}\Eb (T^2)<+\infty$),
 it is enough to prove that $\sup_{d\geq 8}\Eb (T^{\ep})<+\infty$ (resp. $\sup_{d\geq 10}\Eb [(T^{\ep})^2]<+\infty$) for some fixed $\ep.$ 
 
 Now, repeating the proof of (1.12) in \cite{BSZ03}, we obtain
for $k_j=1+Lj$, $j\geq 0$ ($L\geq 1, J\geq 1$ are two fixed integers),
\begin{align}
\Pb\left(T^{\ep}>k_{2J}\right)&\leq \Pb(0\in\D^{\ep})^J+(2J+1)\sum_{k\geq L}k\Pb\left(Z^{\ep}_k=0\right)\notag\\
&\leq {\ep}^J+(2J+1)\sum_{k\geq L}k\Pb\left(Z^{\ep}_k=0\right).
\end{align} 
Using the fact that $\Pb\left(Z^{\ep}_n=0\right)$ decreases with $d\geq 2$ (we delay the proof to the end), let $D\geq 6$, we have
\begin{align}
\Pb[T^{\ep}>k_{2J}]\,(\text{with } d\geq D)&\leq {\ep}^J+(2J+1)\sum_{k\geq L}k\Pb\left(Z^{\ep}_k=0\right)\text{ (when }d=D)\notag\\
&\leq \ep^J+ (2J+1)\text{ const } L^{-\frac{D-5}{2}}.
\end{align}
Choosing a large enough $\ga$ depending on $\ep$, and setting $J=[\ga\log n]$, $L=[\frac{n}{3J}]$ then 
\begin{equation}
\Pb[T^{\ep}>n]\leq c(\log n)^{1+\frac{D-5}{2}}n^{-\frac{D-5}{2}},\quad n\geq 1,\, d\geq D,
\end{equation}
and
\begin{equation}
n\Pb[T^{\ep}>n]\leq c(\log n)^{1+\frac{D-5}{2}}n^{-\frac{D-7}{2}},\quad n\geq 1,\, d\geq D,
\end{equation}
where c depends only on $D$ and $\ep.$ This implies that choosing $D=8$ we get $\sup_{d\geq 8}\Eb{T^{\ep}}<\infty$ and choose $D=10$ we get $\sup_{d\geq 10}\Eb{[{(T^{\ep}})^2]}<\infty.$

Now, in order to finish the proof of Lemma \ref{TL2}, we have to 
prove that $\Pb[Z^{\ep}_n=0]$ decreases with $d\geq 2$.\\
Remark that for $n$ odd $\Pb[Z^{\ep}_n=0]=0$, so we consider $n$ even.
Using characteristic functions, we obtain
\def \Th {\Theta}
\begin{align}
\Pb(Z^{\ep}_n=0)&=\frac{1}{(2\pi)^{d-1}}\int_{-\pi}^{\pi}...\int_{-\pi}^{\pi}\left(\frac{\ep}{d-1}\sum_{i=1}^{d-1}\cos \th_i+1-\ep\right)^nd\th_1...d\th_{d-1}\notag\\
&=\frac{1}{(2\pi)^{d-1}}\int_{-\pi}^{\pi}...\int_{-\pi}^{\pi}\left(\frac{\ep}{d-1}\sum_{i=1}^{d-1}\left(\cos \th_i+\frac{1-\ep}{\ep}\right)\right)^nd\th_1...d\th_{d-1}\notag\\
&=\Eb\left[\left(\frac{1}{d-1}\sum_{i=1}^{d-1}\left(\ep\cos \Th_i+1-\ep\right)\right)^n \right].
\end{align}
where we consider a sequence $\{\Th_i\}_{i=1}^{d-1}$ of i.i.d. random variables having
 uniform distribution $U[-\pi,\pi].$
Now, we consider the function $f(x)=x^n$, $n$ is even, $f$ is a convex function on $\Rb$ and  
$$f\left(\frac{x_1+x_2+...+x_d}{d}\right)\leq\frac{f(x_1)+f(x_2)+...+f(x_d)}{d},\quad\forall x_1,x_2,...,x_d\in\Rb.$$  
For $a_1,a_2,...,a_d\in\Rb$, choose $$x_1=\frac{a_1+a_2+...+a_{d-1}}{d-1},\,x_2=\frac{a_2+a_3+...+a_{d}}{d-1},...,x_d=\frac{a_d+a_1+...+a_{d-2}}{d-1},$$ then we get
\begin{align}
&\left(\frac{a_1+a_2+...+a_d}{d}\right)^n\notag\\
&\leq\frac{1}{d}\left\{\left(\frac{a_1+a_2+...+a_{d-1}}{d-1}\right)^n+\left(\frac{a_2+a_3+...+a_d}{d-1}\right)^n+...+\left(\frac{a_d+a_1+...+a_1}{d-1}\right)^n\right\}.
\end{align}
Now, take $a_i=\ep\cos \Th_i+1-\ep\text{ for }i=1, \cdots, d$ and take the expectation. It comes
\begin{align}
\Eb\left[\left(\frac{1}{d}\sum_{i=1}^d\left(\ep\cos \Th_i+1-\ep\right)\right)^n\right]
\leq \Eb\left[\left(\frac{1}{d-1}\sum_{i=1}^{d-1}\left(\ep\cos \Th_i+1-\ep\right)\right)^n\right].
\end{align}
It means that $\Pb[Z^{\ep}_n=0]$ decreases with $d\geq 2.$ 
\end{proof}

\subsection{Girsanov's transform}
\label{Girsanov}
This section is devoted to the Girsanov's transform connecting $\Pb_{\be}$ and $\Pb_0$ where $\be=\{(\be_1,\be_2,...,\be_m)(y)\}_{y\in\Zb^d}$ is fixed environment. 
%
We begin by introducing several $\si$-algebras. For $n \in \Zb$, let  
$\F_n^Z = \sigma(Z_k, k \leq n)$. For $n \geq 0$, let $\F_n^Y = \sigma(Y_k, 0 \leq k \leq n)$,
$\F_n = \si(\F_n^Z,\F_n^Y)= \si(\F_{-1}^Z, \F_n^Y)$, and 
$\G_n= \si(\F_n^Y,\sigma(Z_k, k \in \Zb))$. We get $\F_n \subset  \G_n$. Moreover $T$ is not
a $(\F_n)$-stopping time,  but is obviously a $(\G_n)$-stopping time, 
so that we can define the $\sigma$-algebra $\G_T$ of the events prior to $T$. Recall that $\E_j=(Y_{j+1}-Y_j).e_1$ and $\{Y_j\notin_k\}$ means that $Y_j$ has been visited exactly $k$ times at time $j$. We define for $n \ge 0$, and $\beta \in ([-1,1]^m)^{\Zb^d}$: 
$$ M_n(\be)=\prod_{j=0}^{n-1}\prod_{k=1}^{m} \left[1+\E_j\be_k(Y_j) 1_{Y_j\notin_k}\right]
\, ,
$$
with the convention the product $\prod_{j=0}^{n-1}(...)=1$ and $M_n(\be)=1$ for $n=0$.
 \begin{lem}
\label{densite} For any $\beta \in ([-1,1]^m)^{\Zb^d}$, $d \geq 6$, $n \ge 0$, 
$$ M_n(\be)  = \frac{d\mathbb{P}_{\be}|_{\F_n}}{d\mathbb{P}_0|_{\F_n}} \, , \, \, 
 M_n(\be) = \frac{d\mathbb{P}_{\be}|_{\G_n }}{d\mathbb{P}_0|_{\G_n}} \, , \, \, 
 M_T(\be)  = \frac{d\mathbb{P}_{\be}|_{\G_T}}{d\mathbb{P}_0|_{\G_T}} \, .
 $$
 \end{lem}
 
\begin{proof}
Since $ \F_n \subset \G_n$, $M_n(\beta)$ is $ \F_n$-measurable, and $T$ is a finite $(\G_n)$-stopping time,
it is enough to prove that $M_n(\be) = \frac{d\mathbb{P}_{\be}|_{\G_n }}{d\mathbb{P}_0|_{\G_n}}$. 
Let $A \in \F_{-1}^Z$, $y_1,..., y_n \in (\Zb^d)^n$, and $B \in \sigma(Z_{n+k}-Z_n, k \ge 0)$ be fixed.
Since $(Z_{n+\cdot}-Z_n)$ is independent from $\F_n$, we get:
$$\Pb_{\be}(A, Y_0=0, Y_1=y_1,..., Y_n=y_n, B)
	=   \Pb_{\be}(A, Y_0=0, Y_1=y_1,..., Y_n=y_n) \Pb_{\be}(B).
$$
Note that the law of $Z$ does not depend on $\be$, so that $\Pb_{\be}(B)=\Pb_{0}(B)$. Now by the definition of $m-$ERWRC, 
$$
\mathbb{P}_{\be}[Y_n=y_n \left| A, Y_0=0, Y_1=y_1,..., Y_{n-1}=y_{n-1} \right.]
 = \frac{1}{2d}\prod_{k=1}^{m}\left[1+{\ep}_{n-1}\be_k(y_{n-1})1_{y_{n-1}\notin_k}\right],\\
$$ 
where $\ep_{n-1}=(y_{n}-y_{n-1}).e_1$.
Then we get by induction that for any $\beta \in ([-1,1]^m)^{\Zb^d}$,
\begin{align*}
\mathbb{P}_{\be}[A, Y_0=0, Y_1=y_1,..., Y_n=y_n] 
& = \left(\frac{1}{2d}\right)^n\prod_{j=0}^{n-1}\prod_{k=1}^{m}\left[1+{\ep}_j\be_k(y_j)1_{y_j\notin_k}\right]
\mathbb{P}_{\be}[A]
\\
& = \left(\frac{1}{2d}\right)^n\prod_{j=0}^{n-1}\prod_{k=1}^{m}\left[1+{\ep}_j\be_k(y_j)1_{y_j\notin_k}\right]
\mathbb{P}_{0}[A]
\, ,
\end{align*}
where the last equality comes from the fact that $A\in \F_{-1}^Z$.  
Hence, 
$$\frac{\mathbb{P}_{\be}[A, Y_0=0, Y_1=y_1,..., Y_n=y_n, B]}
{\mathbb{P}_0[A,Y_0=0,Y_1=y_1,...,Y_n=y_n,B]}
= \prod_{j=0}^{n-1}\prod_{k=1}^{m}\left[1+{\ep}_j\be(y_j)1_{y_j\notin_k}\right] \, .
$$
We have just proved that for all 
$A \in \F_{-1}^Z$, $y_1,..., y_n \in (\Zb^d)^n$, and $B \in \sigma(Z_{n+k}-Z_n, k \ge 0)$,
$$ \mathbb{P}_{\be}[A,Y_0=0,Y_1=y_1,...,Y_n=y_n,B]
= \Eb_{0}[1_A \, 1_{Y_0=0,Y_1=y_1,..,Y_n=y_n} \, 1_B \, M_n(\be)].
$$
The result follows since $\G_n = \si(\F_{-1}^Z,\F_n^Y,\sigma(Z_{n+k}-Z_n, k \ge 0))$. 
\end{proof}
\subsection{Existence of the speed.}
\subsubsection{$e_1-$ exchangeable and stationary environment}
We begin with some notations used throughout the section. 
For $z \in (\Zb^{d-1})^{\Zb}$, and $k,l \in \Zb, k\leq l$, $z_{[k,l]}:=(z_k,z_{k+1},\cdots,z_l)$. The expectation w.r.t. the
law $\Qb$ of the environment is still denoted by $\Qb$.  
We also use the 
notation $\hat{P}(\cdot)=P(\cdot|0\in\D)$, and for $\beta$ fixed, $\hat{\Pb}_{\be}(\cdot)=\Pb_{\be}(\cdot|0\in\D)$.
Since $\Pb_{\be}(0\in\D)$ does not depend on $\be$, we get $\hat{P}(\cdot)=\Qb(\hat{\Pb}_{\be}(\cdot))$.
Let $A$ be any Borel set of $(\Zb^{d})^{\Nb}$, then 
\begin{align}
& \hat{P}(Y_{T+.}-Y_T\in A)=\Qb[\hat{\Pb}_{\be}(Y_{T+.}-Y_T\in A)]\notag\\
&=\sum_{k\geq 1}\sum_{z_{[1,k]}}\sum_{x\in\Zb} \Qb[\hat{\Pb}_{\be}(Y_{k+.}-Y_k\in A|
		T=k,Z_{[1,k]}=z_{[1,k]},X_k=x) 
\notag\\
& \hspace*{5cm} \times
		\hat{\Pb}_{\be}(X_k=x|T=k,Z_{[1,k]}=z_{[1,k]})] 
		\hat{P}(T=k,Z_{[1,k]}=z_{[1,k]}) \, .
\end{align}

Note that by the definition of the cut times, the trajectory of $Y$ between $T_n$ and $T_{n+1}-1$ does not intersect the
trajectory of $Y$ before $T_n$. 
Hence 
$\hat{\Pb}_{\be}(Y_{k+.}-Y_k\in A|T=k,Z_{[1,k]}=z_{[1,k]},X_k=x)$ depends only
on $\{\be(.,z)\}_{z\notin z_{[1,k]}}$, while 
$\hat{\Pb}_{\be}(X_k=x|T=k,Z_{[1,k]}=z_{[1,k]})$
depends only on $\{\be(.,z)\}_{z\in z_{[1,k]}}$.
$z_{[1,k]}$ and $x\in\Zb$ being given, we consider 
the mapping $\de:\Zb^{d}\to\Zb^{d}$ defined by:
$$ \forall (u,v)\in\Zb\times \Zb^{d-1}, \, \, 
\de (u,v)= \left\{ \begin{array}{ll}
				(u,v) & \mbox{ if } v \in z_{[1,k]} \, , 
				\\
				(u-x,v) & \mbox{ if } v\notin z_{[1,k]} \, .
				\end{array}
				\right .
				$$
It follows from the preceding remark that:
\begin{align*}
&  \hat{\Pb}_{\de \be}(Y_{k+.}-Y_k\in A|T=k,Z_{[1,k]}=z_{[1,k]},X_k=x)
\\
& =\hat{\Pb}_{\th_{(-x,0)} \be}(Y_{k+.}-Y_k\in A|T=k,Z_{[1,k]}=z_{[1,k]},X_k=x)
\\
& = \hat{\Pb}_{\th_{( 0,z_k)} \be}(Y_{.}\in A|T_{-1}=-k,Z_{[-k,-1]}=\bar{z}_{[-k,-1]}) \, ,
\end{align*} 
where $\th_{(x,z)}\be(u,v)=\be(u+x,v+z)$,  and $\bar{z}_{[-k,-1]}=(-z_k,z_1-z_k,\cdots,z_{k-1}-z_k)$. 
Moreover, 
 $$ \hat{\Pb}_{\de \be}(X_k=x|T=k,Z_{[1,k]}=z_{[1,k]})= \hat{\Pb}_{\be}(X_k=x|T=k,Z_{[1,k]}=z_{[1,k]})
 \, .
 $$

The random environment being $e_1$-exchangeable, $\de(\be)$ has the same law as $\be$. Hence, 
{\allowdisplaybreaks
\begin{align}
& \hat{P}(Y_{T+.}-Y_T\in A)
\notag\\
&=\sum_{k\geq 1}\sum_{z_{[1,k]}}\sum_{x\in\Zb}
\Qb[\hat{\Pb}_{\de\be}(Y_{k+.}-Y_k\in A|T=k,Z_{[1,k]}=z_{[1,k]},X_k=x)
\notag\\
& \hspace*{5cm}  \times \hat{\Pb}_{\de\be}(X_k=x|T=k,Z_{[1,k]}=z_{[1,k]})] 
\hat{P}(T=k,Z_{[1,k]}=z_{[1,k]}) 
\notag\\
&=\sum_{k\geq 1} \sum_{z_{[1,k]}} \sum_{x\in\Zb}
\Qb[\hat{\Pb}_{\th_{(0,z_k)}\be}(Y_{.}\in A|T_{-1}=-k,Z_{[-k,-1]}=\bar{z}_{[-k,-1]})
\notag\\
& \hspace*{5cm}  \times\hat{\Pb}_{\be}(X_k=x|T=k,Z_{[1,k]}=z_{[1,k]})] \hat{P}(T=k,Z_{[1,k]}=z_{[1,k]}) 
\notag\\
&=\sum_{k\geq 1} \sum_{z_{[1,k]}} 
\Qb[\hat{\Pb}_{\th_{(0,z_k)}\be}(Y_{.}\in A|T_{-1}=-k,Z_{[-k,-1]}=\bar{z}_{[-k,-1]})]
\hat{P}(T=k,Z_{[1,k]}=z_{[1,k]}).
\end{align}}
Using the stationarity of the environment, we get then 
{\allowdisplaybreaks
\begin{align}
& \hat{P}(Y_{T+.}-Y_T\in A)
\notag\\
&=\sum_{k\geq 1}\sum_{z_{[1,k]}}
\Qb[\hat{\Pb}_{\be}(Y_{.}\in A|T_{-1}=-k,Z_{[-k,-1]}=\bar{z}_{[-k,-1]})]
\hat{P}(T_{-1}=-k,Z_{[-k,-1]}=\bar{z}_{[-k,-1]})
\notag\\
&=\hat{P}(Y_.\in A).
\end{align}}
Now, set $H_n=X_{T_{n}}-X_{T_{n-1}}$ for $n\geq 1$. We have just seen that the sequence $\{H_n\}_{n\geq 1}$ is  stationary 
under $\hat{P}$. Furthermore, $\hat{E}|H_n|\leq \hat{E}T<\infty$ for $d\geq 6$. By the ergodic theorem,  $\hat{P}-a.s.$
$$\lim_{n\to\infty}\frac{H_1+H_2+...+H_n}{n}=\hat{E}(H_1|\F_H),$$ where $\F_H$ is the 
$\si$-algebra generated by the invariant sets  of the sequence $\{H_n\}.$ 
Therefore $\lim_{n\to\infty}\frac{X_{T_n}}{n}=\hat{E}(X_T|\F_H)$.
On the other hand, we also have $\hat{P}-as$ $\lim_{n\to\infty}\frac{T_n}{n}=\hat{E}(T)$, 
 so that $\hat{P}-as$, $V:=\lim_{n\to\infty}\frac{X_n}{n}$ exists for $d \geq 6$, and
$$V =\frac{\hat{E}(X_T|\F_H)}{\hat{E}(T)}.$$
\subsubsection{i.i.d random environment.}\label{seciid}

We consider now the case of an i.i.d environment with $m$ cookies. In this situation, we can prove that the speed is deterministic. To this end, we construct  an ergodic dynamical system on which 
the  $m-$ERWRC is defined. Let $\mu$ be the law of $\be = (\be_1,\be_2,...,\be_m)(0)  \in [-1,1]^m$. 

We consider the probability space 
$$W:= \Ga\ti (\Zb^{d-1})^{\Zb}\ti\{0,1\}^{\Zb}\ti \{\{0,1\}^m\}^{\Zb}\text{ where }\Ga=([-1,1]^{m})^{\Zb},$$
endowed with  the  probability semi-product $P_s:=\Qb_s \ti \Pb_{\ga}$, where $\Qb_s=\mu^{\otimes \Zb}$ and for $\ga \in \Ga$,
$$\Pb_{\ga}=q^{\otimes\Zb} \otimes p_{1}^{\otimes\Zb}
 \otimes \underset{n\in \Zb}{\bigotimes}\underset{1\leq k\leq m}{\bigotimes}p_{kn}(\ga) \, , 
 $$
where 
\begin{itemize}
\item $q$ is the law of the increments of $Z$, 
\item $p_1$ is a Bernoulli distribution of parameter 1/2,
\item  $p_{kn}(\ga)$ is a Bernoulli distribution with $p_{kn}\{1\}=\frac{1+\ga_n(k)}{2},\ p_{kn}\{0\}=\frac{1-\ga_n(k)}{2}.$
\end{itemize} 
Now, we take $w=(\ga,u,l,h)\in W$ with $\ga\in\Ga$, $u\in{(\Zb^{d-1})}^{\Zb}$, $l\in\{0,1\}^{\Zb}$, $h\in \{\{0,1\}^m\}^{\Zb}$. 
For  $n\in\Zb$, let  $(\be_n, I_n, \ze_n, \xi_n)$ be the canonical process on $W$: 
$$ \be_n(w)=\ga_n \in [-1,1]^m \, , \, \, I_n(w)=u_n \in \Zb^{d-1} \, , \, \, \ze_n(w)=l_n \in  \{0,1\} \, , \,\, 
\xi_{n,j}(w)=h_{n,j} \in  \{0,1\} \, .
$$ 
From $(I_n)_{n \in \Zb}$, we define $Z, \tilde{Z}$ as follows:
$$Z_k=
\begin{cases}
I_1+...+I_k &\mbox{ if } k>0 \, , \\
0 &\mbox{ if }k=0 \, , \\
-(I_{k+1}+...+I_0)&\mbox{ if } k<0 \, ,  
\end{cases}\quad\quad \quad 
\et_k:=1_{Z_k=Z_{k+1}}.
$$
Set $U(k):=\inf\{n, \sum_{i=0}^{n-1}(1-\et_i)=k\}$ and $\tilde{Z}_k:=Z_{U(k)}.$ It is clear that this definition of $Z, \tilde{Z}$ satisfies \eqref{lienZ-tildeZ}.
 Once $Z$ is defined, we construct the horizontal part's
increment $\E_i = X_{i+1}-X_i \in \{-1,0,1\}$ for $i \geq 0$, as follows. 
Set $Y_0=0$ and assume that $(Y_0,...,Y_i)$ have been constructed. Then,
\begin{itemize}
\item  On the event $\{Y_i\notin_k\}$ ($1 \leq k \leq m$) (i.e. $Y_i$ has been exactly visited $k$ times at time $i$), 
$$\E_i = (2 \xi_{n_1(Y_0,...,Y_i),k} -1) \, 1_{Z_i=Z_{i+1}} \, , 
$$ 
where $n_1(Y_0,...,Y_i) = \inf \{n \leq i, \mbox{ such that } Y_n=Y_i\}$.
\item  On the event $\{Y_i \in^m \}$ (i.e. $Y_i$ has been visited more than $m$ times at time $i$), 
$$\E_i = (2 \ze_i -1) \, 1_{Z_i=Z_{i+1}} \, .
$$
\end{itemize}
It is proved similarly as in Section \ref{contructES} about the construction $m-$ERWRC to have that the construction of $Y$ above satisfies, 
\begin{align*}
P_{\ga}(Y_{n+1}-Y_n=\pm e_1|\F^{Y}_n, Y_n\notin_{k})=&\frac{1\pm \ga_k(n_1(Y_0,...,Y_n))}{2d}
\text{ for  }k\leq m,\\
P_{\ga}(Y_{n+1}-Y_n=\pm e_i|\F^{Y}_n, Y_n\notin_{k})=&\frac{1}{2d}\text{ for }i>1 \text{ or }k>m. 
\end{align*}

\begin{lem}\label{bodecungluat}
Under $P$,  the sequence   $(Y_n)_{n\geq 0}$ is an $m$-ERWRC with i.i.d environment $\be=(\be(y))_{y\in \Zb^d}$ of common law
$\mu$.
\end{lem}
\begin{proof}
We begin with giving an expression for the law of the   $m$-ERWRC with i.i.d environment $\be$.
Fix $y_0=0, y_1,..., y_n \in \Zb^d$ and  set $\ep_i=(y_{i+1}-y_i).e_1 \in\{0,\pm 1\}.$ 
Then, for an  $m$-ERWRC with i.i.d environment $\be$,  we have 
\begin{align*}
&\Qb \Pb_\be[Y_0=y_0, Y_1=y_1,...,Y_n=y_n]\\
&=\Qb\left[\left(\frac{1}{2d}\right)^{n}\prod_{i=0}^{n-1}\prod_{k=1}^{m}(1+\be_k(y_i)\ep_i \, 1_{y_i\notin_k})\right] \, .
\end{align*}
We decompose the first product according to the value of the first visit to $y_i$.
\begin{align*}
& \Qb \Pb_\be[Y_0=y_0, Y_1=y_1,...,Y_n=y_n]\\
&=\left(\frac{1}{2d}\right)^{n} \Qb\left\{\prod_{n_1=0}^{n-1}
\prod_{k=1}^{m} \prod_{j=n_1}^{n-1 }\left[1+1_{y_{n_1} \notin} \, \be_k(y_{n_1}) \, \ep_j \, 1_{y_j=y_{n_1}} 
\, 1_{y_{j}\notin_k} \right] \right\}\\
&=\left(\frac{1}{2d}\right)^{n}\prod_{n_1=0}^{n-1} \Qb\left\{\prod_{k=1}^{m}\prod_{j=n_1}^{n-1}
\left[1+1_{y_{n_1} \notin} \, \be_k(y_{n_1}) \, \ep_j \, 1_{y_j=y_{n_1}} \, 1_{y_{j} \notin_k}\right]\right\}.
\end{align*}
The last equation comes from the independence of the random variables $\be_k(y_i)$ for $y_i\notin.$
On the other hand, using the construction above,  
\begin{align*}
&P_s[Y_0=y_0, Y_1=y_1,...,Y_n=y_n]=\Qb_s \Pb_\ga[Y_0=y_0, Y_1=y_1,...,Y_n=y_n]\\
&=\left(\frac{1}{2d}\right)^{n} {\Qb_s}\left\{\prod_{n_1=0}^{n-1}\prod_{k=1}^{m}\prod_{j=n_1}^{n-1}
\left[1+1_{y_{n_1}\notin}\ga_k(n_1) \, \ep_i \, 1_{y_j=y_i} \, 1_{y_i\in_k} \right]\right\}\\
&=\left(\frac{1}{2d}\right)^{n}\prod_{n_1=0}^{n-1} {\Qb_s}\left\{\prod_{k=1}^{m}\prod_{j=n_1}^{n-1}
\left[1+1_{y_i\notin} \, \ga_k(n_1) \, \ep_i \, 1_{y_j=y_i} \, 1_{y_i\in_k}\right]\right\}.
\end{align*}
This finishes the proof of Lemma $\ref{bodecungluat}$ since $\{1_{y_{n_1} \notin} \be(y_{n_1})\}_{n_1=0,...,n-1}$ and 
$\{1_{y_{n_1} \notin}\ga(n_1)\}_{n_1=0,...,n-1}$  are two sequences of i.i.d. random vectors with common law $\mu$. 
\end{proof}

Now, we denote by $(\th_k)_{k\in \Zb}$ the canonical shift on $W$, i.e. $\th_k(w.)=(w_{k+.})$. 
We set 
$$\hat{W} = \Ga \times\left[ (\Zb^{d-1})^{\Zb} \cap\{0\in\D\}\right] \times \{0,1\}^{\Zb} \times (\{0,1\}^m)^{\Zb}
\, .
$$
On $\hat{W}$ we define $\hat{\th}:=\hat{\th}_1=\th_T$ and $\hat{P}_s(\cdot)=P_s(\cdot|0\in\D).$ 
\begin{lem}
$(W,\th,P_s) $ is an ergodic system. As a consequence, $(\hat{W},\hat{\th},\hat{P}_s) $ is also an ergodic system.
\end{lem}
\begin{proof}
The idea of proof comes from $\cite{BSZ03}$.
Firstly, we prove that $\th$ is a measure-preserving transformation.
Consider a measurable set $A\times B$ of $W$, where $A \subset \Ga$, and $B \subset (\Zb^{d-1})^{\Zb} \times \{0,1\}^{\Zb}
\times (\{0,1\}^m)^{\Zb}$ .
We have that 
{\allowdisplaybreaks
\begin{align*}
&\th_k\circ P_s(A\times B)=P_s(\th_k^{-1}A\times\th_k^{-1}B)\\
&=\int_{\th_k^{-1}A}\Pb_{\ga}(\th_k^{-1}B)d\Qb=\int_{\th_k^{-1}A}\Pb_{\th_k\ga}(B)d\Qb_s\\
&=\int_{A} \, \Pb_{\ga}(B) \, (\th_k^{-1}\Qb_s)(d\ga) =\int_{A}\Pb_{\ga}(B)\Qb_s(d\ga)
\\
&=P_s(A\times B).
\end{align*}
}
Now, we prove that $\th$ is ergodic.
Let A be a measurable subset of W, invariant under $\th$ and $\ep>0.$
There exists an integer $m_{\ep}>0$ and a measurable subset $A_{\ep}$ depending only on $(w_m)_{|m|\leq m_ {\ep}}$ such that
$$ | E_{P_s}[1_A-1_{A_{\ep}}] | \leq\ep.$$
Then, for $L\geq 0,$ 
$$ P_s(A)=\Eb_{P_s}[1_A1_A\circ\th_L]= \Eb_{P_s}[1_{A_{\ep}}1_{A_{\ep}}\circ\th_L]+c_{\ep},$$
with $|c_{\ep}|\leq 2\ep.$

Because that $p_{kn}(\ga)$ depends only on $\ga_n$, we prove that the sequence $(\ga_n, I_n, \ze_n, \xi_n)_{n\in \Zb}$ is the sequence of independent variables under $P_s.$
Indeed, let $i<j,\,i,j\in\Zb$, we take two measurable sets $A_i\times B_i$ and $A_j\times B_j$, where $A_i, A_j \subset [0,1]^m$, and $B_i, B_j \subset \Zb^{d-1} \times \{0,1\}
\times \{0,1\}^m$.
We have 
\begin{align*}
&P_s\left\{[(\ga_i, I_i, \ze_i, \xi_i)\in A_i\times B_i]\bigcap[(\ga_j, I_j, \ze_j, \xi_j)\in A_j\times B_j]\right\}\\
&=\int_{\Ga}\Qb_s(d\ga)1_{\ga_i(\ga)\in A_i, \ga_j(\ga)\in A_j}\Pb_{\ga}\left\{[(I_i, \ze_i, \xi_i)\in B_i]\bigcap[(I_j, \ze_j, \xi_j)\in B_j]\right\}\\
&=\int_{\Ga}\Qb_s(d\ga)1_{\ga_i(\ga)\in A_i, \ga_j(\ga)\in A_j}\Pb_{\ga}\left\{(I_i, \ze_i, \xi_i)\in B_i\right\}\Pb_{\ga}\left\{[(I_j, \ze_j, \xi_j)\in B_j\right\}\\
&=\int_{\Ga}\Qb_s(d\ga)1_{\ga_i(\ga)\in A_i}\Pb_{\ga}\left\{(I_i, \ze_i, \xi_i)\in B_i\right\}.\int_{\Ga}\Qb_s(d\ga)1_{\ga_j(\ga)\in A_j}\Pb_{\ga}\left\{(I_j, \ze_j, \xi_j)\in B_j\right\}\\
&=P_s\left\{[(\ga_i, I_i, \ze_i, \xi_i)\in A_i\times B_i]\right\}.P_s\left\{[(\ga_j, I_j, \ze_j, \xi_j)\in A_j\times B_j]\right\}
\end{align*} 
So, for $L>2m_{\ep}$, we get 
$\Eb_{P_s}[1_{A_{\ep}}1_{A_{\ep}}\circ\th_L] =  P_s(A_{\ep}) P_s(A_{\ep}\circ\th_L) = P_s(A_{\ep})^2$. 
Therefore $$|P_s(A)-P_s(A)^{2}|\leq|P_s(A)-P_s(A_{\ep})^{2}|+2\ep\leq 4\ep.$$ Letting $\ep$ tend to $0$, we have that $P_s(A)=0$ or $1$.
\end{proof}
\begin{lem} Let $Y$ is a $m-$ERWRC such that the environment cookie is i.i.d., $X$ is the horizontal component $X_n=Y_n\cdot e_1.$ 
For any $d \geq 6$ then $P-as$, $\lim_{n\to\infty}\frac{X_n}{n}=v(\Qb):=\frac{\hat{E}(X_T)}{\hat{E}(T)}$. 
\end{lem}
\begin{proof}
The existence of the limit, the fact that it is deterministic and the expression of $v(\Qb)$ for $d\geq 6$ follow from the ergodicity of  $(\hat{W},\hat{\th},\hat{P}_s)$, and the integrability of $T$ w.r.t $\hat{P}_s$ when $d \geq 6$. 
\end{proof}

\subsection{Monotonicity and differentiability of the speed.} 
Now, we prove that the expectation $v(\Qb)=\hat{E}[V]=\frac{\hat{E}(X_T)}{\hat{E}(T)}$ is increasing in $\Qb$.\\
Consider $\be_1=\{\be_1(y)\}_{y\in\Zb^{d}},\be_2=\{\be_2(y)\}_{y\in\Zb^{d}}$ defined 
on $(\Om,\A,Q)\to\Bb=([-1,1]^m)^{\Zb^{d}}$ such that $Q(\be_1\leq\be_2)=1$. 
It is proved in D. Aldous and R. Lyons \cite{AL07}, that  if there exists a monotone coupling of 
$\Qb_1$ and $\Qb_2$, then there also exists a stationary  monotone coupling of $\Qb_1$ and $\Qb_2$, as soon as
$\Qb_1$ and $\Qb_2$ are stationary.

 Therefore we can suppose that $\{(\be_1,\be_2)(y)\}_{y\in\Zb^{d}}$ is stationary.
 Set $\be_t(y)=(1-t)\be_1(y)+t\be_2(y)$ for $t\in[0,1]$. $\be_t=\{\be_t(y)\}_{y\in\Zb^{d}}$ is a stationary 
 environment . Consider 
 \begin{align}\label{vitesst}
  f(t):=\frac{\Qb\Eb_{\be_t}(X_T \, 1_{0\in\D})}{E(T \, 1_{0\in\D})}.
\end{align} 
 Note that $\be_t$ is not necessarily exchangeable, so that we can not assert that $f(t)$ is the mean
 of the speed of the ERW in the random environment $\be_t$. Nevertheless, $\be_1$ and $\be_2$ being 
 exchangeable, we get
 $f(0)=v(\Qb_{1}),\ f(1)=v(\Qb_{2})$, so that it is enough to prove that $f(t)$ is increasing in $t$.
First of all, we need the Girsanov's transform. We have
$$ M_n(\be_t)=\prod_{j=0}^{n-1} [1+\E_j\be_t(Y_j)1_{Y_j\notin^{m}}]  \, ,
$$
where $Y_j\notin^{m}$ denotes the event that $Y_j$ has not been visited more than $m-1$ times before time $j$. 
As in section \ref{Girsanov}, we have Girsanov's transforms:
$$ \frac{{dP_{\be_t}}|_{\F_n}}{{dP_{0}}|_{\F_n}} 
=\frac{{dP_{\be_t}}|_{\G_n}}{{dP_{0}}|_{\G_n}} = M_n(\be_t) \, , \, \, 
\frac{{dP_{\be_t}}|_{\G_T}}{{dP_{0}}|_{\G_T}} = M_T(\be_t) \, .$$
\subsubsection{Differentiability of $f(t)$.}
\label{ft}
We begin by giving another expression of the
numerator in \eqref{vitesst}.
  \begin{lem}\label{20}
For $n \ge 1$, 
 then
\begin{align}\label{NumV}
\Eb_{\be_t}(X_T1_{0\in\D})&=\Eb_{\be_t}\left[\sum_{j=0}^{T-1} \be_t(Y_j) \, 1_{0\in\D} \, 1_{Y_j\notin^{m}} \, 1_{Z_j=Z_{j+1}}\right]\notag\\&=\Eb_{0}\left[\sum_{j=0}^{T-1} \be_t(Y_j) \, 1_{0\in\D} \, 1_{Y_j\notin^{m}} \, 1_{Z_j=Z_{j+1}} \, M_T(\be_t)\right].
\end{align}

\end{lem}
\begin{proof}
Observe that 
\begin{align}\label{a}
\Pb_{\be_t}[\E_j=\pm1|\G_j]&=\frac{1\pm\be_t(Y_j)}{2}1_{Y_j\notin^m}1_{Z_j=Z_{j+1}}+\frac{1}{2}1_{Y_j\in^m}1_{Z_j=Z_{j+1}}\notag\\
&=\left(\frac{1}{2}\pm\frac{\be_t(Y_j)}{2}1_{Y_j\notin^m}\right)1_{Z_j=Z_{j+1}}.
\end{align}
Hence,
$$ \Eb_{\be_t}(X_T1_{0\in\D})=\Eb_{\be_t}(\sum_{j=0}^{+\infty}\E_j1_{T>j}1_{0\in\D})=\sum_{j=0}^{+\infty}\Eb_{\be_t}(\E_j1_{T>j}1_{0\in\D}) \, ,$$
where the last equality follows from the integrability of $T$ w.r.t $\hat{\Pb}$ for $d\ge 6$.
Note that  $\{0\in\D\}\text{ and }\{T>j\}$ belong to $\G_j$. Therefore,
{\allowdisplaybreaks
\begin{align*}
&\Eb_{\be_t}(\E_j1_{T>j}1_{0\in\D})
\\
&=\Eb_{\be_t}\left[1_{T>j} \, 1_{0\in\D} \, \Pb_{\be_t}(\E_j=1|\G_j)\right]
-\Eb_{\be_t}\left[1_{T>j} \, 1_{0\in\D} \, \Pb_{\be_t}(\E_j=-1|\G_j)\right]\\
&=\Eb_{\be_t}\left[\frac{1+\be_t(Y_j)}{2}1_{T>j}1_{0\in\D}1_{Y_j\notin^{m}}1_{Z_j=Z_{j+1}}\right]-\Eb_{\be_t}\left[\frac{1-\be_t(Y_j)}{2}1_{T>j}1_{0\in\D}1_{Y_j\notin^{m}}1_{Z_j=Z_{j+1}}\right]\\
&=\Eb_{\be_t}\left[ \be_t(Y_j) \, 1_{T>j} \, 1_{0\in\D} \, 1_{Y_j\notin^{m}} \, 1_{Z_j=Z_{j+1}}\right].
\end{align*}
Thus,
\begin{align*}
\Eb_{\be_t}(X_T1_{0\in\D})&=\sum_{j=0}^{+\infty}\Eb_{\be_t}\left[ \be_t(Y_j) \, 1_{T>j}1_{0\in\D} \, 1_{Y_j\notin^{m}} \, 1_{Z_j=Z_{j+1}}\right]\\
&=\Eb_{\be_t}\left[\sum_{j=0}^{T-1} \be_t(Y_j) \, 1_{0\in\D} \, 1_{Y_j\notin^{m}} \, 1_{Z_j=Z_{j+1}}\right]\\
&=\Eb_{0}\left[\sum_{j=0}^{T-1} \be_t(Y_j) \, 1_{0\in\D} \, 1_{Y_j\notin^{m}} \, 1_{Z_j=Z_{j+1}} \, M_T(\be_t)\right].
\end{align*}
}

This proves the first equality. The second one follows from the fact that $\sum_{j=0}^{T-1} \be_t(Y_j) \, $ $1_{0\in\D} \, 1_{Y_j\notin^{m}} \, 1_{Z_j=Z_{j+1}}$ is $\G_T$-measurable, and Lemma \ref{densite}.
\end{proof}

We turn now to the derivative of $f(t)$. We study now the sign of the derivative on the set of bounded environment, and from now
on we assume that for $i=1,2$, $|\be_i(y)| \leq \si <1$ a.s. for any $y$ of $\Zb^d$ where $\si$ is a constant in $(0,1)$. 
\begin{lem}\label{Dhspeed}
For $d \ge 8$, the function $t \in [0,1] \to \Qb[\Eb_{\be_t}(X_T \, 1_{0\in\D})]$ is differentiable and, 
{\allowdisplaybreaks
\begin{align}\label{dht}
&E(T1_{0\in\D}) .\frac{\partial f}{\partial t}(t)
 =\frac{\partial}{\partial t}  \Qb[\Eb_{\be_t}(X_T1_{0\in\D})] \notag
\\
&=\Qb \Eb_{\be_t}\left[\sum_{j=0}^{T-1} (\be_2-\be_1)(Y_j) \, 1_{0\in\D} \, 1_{Y_j\notin^{m}} \, 1_{Z_j=Z_{j+1}}\right]\notag\\
& \quad  + \Qb \Eb_{\be_t}\left[\sum_{j=0}^{T-1} \be_t(Y_j) \, 1_{0\in\D} \, 1_{Y_j\notin^{m}} \, 1_{Z_j=Z_{j+1}}
	\sum_{i=1}^{T-1}\frac{(\be_2-\be_1)(Y_i)\E_i1_{Y_i\notin^m}}{1+\be_t(Y_i)\E_i}
	\, 1_{Z_i=Z_{i+1}} \right].
\end{align}
}
\end{lem}
\begin{proof}
We have $M_T(\be_t)=\prod_{j=0}^{T-1} \left[1+\E_j\be_t(Y_j) 1_{Y_j\notin^m}\right]$ then
\begin{align*}
\frac{\partial}{\partial t}M_T(\be_t)&=\left(\sum_{j=0}^{T-1}\frac{(\be_2-\be_1)(Y_j)\E_j}{1+\be_t(Y_j)\E_j}1_{Y_j\notin^m}\right)M_T(\be_t)\\
&=\left(\sum_{j=0}^{T-1}\frac{(\be_2-\be_1)(Y_j)\E_j}{1+\be_t(Y_j)\E_j}1_{Y_j\notin^m}1_{Z_j=Z_{j+1}}\right)M_T(\be_t)
\end{align*}
the last equality is followed by the fact that $Z_j=Z_{j+1}$ when $\E_j\neq 0.$
Set 
\begin{align*}
N_T(t):=\sum_{j=0}^{T-1} \be_t(Y_j) \, 1_{0\in\D} \, 1_{Y_j\notin^{m}} \, 1_{Z_j=Z_{j+1}}.
\end{align*}
Then
\begin{align*}
\frac{\partial}{\partial t}N_T(t)=\sum_{j=0}^{T-1} (\be_2-\be_1)(Y_j) \, 1_{0\in\D} \, 1_{Y_j\notin^{m}} \, 1_{Z_j=Z_{j+1}}.
\end{align*}
and
\begin{align}\label{btU}
\quad\quad\quad U_T(\be_t):&=\frac{\partial}{\partial t}\left[N_T(t)M_T(\be_t)\right]=\frac{\partial}{\partial t}N_T(t)M_T(\be_t)+N_T(t)\frac{\partial}{\partial t}M_T(\be_t)\notag\\
&=\sum_{j=0}^{T-1} (\be_2-\be_1)(Y_j) \, 1_{0\in\D} \, 1_{Y_j\notin^{m}} \, 1_{Z_j=Z_{j+1}}M_T(\be_t)\notag\\
&\quad\quad\quad+N_T(t)\left(\sum_{j=0}^{T-1}\frac{(\be_2-\be_1)(Y_j)\E_j}{1+\be_t(Y_j)\E_j}1_{Y_j\notin^m}1_{Z_j=Z_{j+1}}\right)M_T(\be_t).
\end{align}

We have 
\begin{equation}\Qb\Eb_{0}\left[N_T(t)M_T(\be_t)\right]=\Qb\Eb_{0}[N_T(0)M_T(\be_0)]+\Qb\Eb_{0}\left[\int_{0}^{t}U_T(\be_x)dx\right].
\end{equation}
Since $N_T(t)\leq T1_{0\in\D}, \frac{\partial}{\partial t}N_T(t)\leq 2T1_{0\in\D} $ and $\left|\frac{\E_j}{1+x\E_j}\right|\leq\frac{1}{1-\si},\forall x\leq \si,$ we get
\begin{align*}
&\int_{0}^{t}\Qb\Eb_{0} |U_T(\be_x)|dx\leq 2\int_{0}^{t}\Qb\Eb_{0}(T1_{0\in\D}M_T(\be_x))dx+\frac{2\si}{1-\si}\int_{0}^{t}\Qb\Eb_{0}(T^21_{0\in\D}M_T(\be_x))dx\\
&=2\int_{0}^{t}\hat{\Eb}_0(T)dx+\frac{2\si}{1-\si}\int_{0}^{t}\Qb\Eb_{\be_x}(T^21_{0\in\D})dx=2\int_{0}^{t}\hat{\Eb}_0(T)dx+\frac{2\si}{1-\si}\int_{0}^{t}\Qb\Eb_{0}(T^21_{0\in\D})dx,\\
&(\text{ since }T\text{ and }\{0\in\D\}\text{ belong to }\si (Z), \text{ then they do not depend on }x )\\
&=2t\hat{\Eb}_0 T+\frac{2t\si}{1-\si} \hat{\Eb}_0(T^2)<+\infty \text{ when }\hat{\Eb}_0(T^2)<+\infty.
\end{align*}
It follows from Lemma \ref{TL2} that 
$\hat{\Eb}_0(T^2) < + \infty$ for $d\geq 8$. 
Fubini's theorem leads then to 
\begin{equation}\Qb\Eb_{0}\left[N_T(t)M_T(\be_t)\right]=\Qb\Eb_{0}[N_T(0)M_T(\be_0)]+\int_{0}^{t}\Qb\Eb_{0}\left[U_T(\be_x)\right]dx.
\end{equation}
Now, we prove that $\Qb\Eb_{0}\left[U_T(\be_x)\right]$ is continuous in $x\in[0,1]$. To this end, we recall a general
result about uniform integrability of positive random variables (see for instance  
Theorem 5 page 189 in Shiryaev \cite{Shi96}).

\begin{lem}\label{bd1}
Let $J$ be an interval of $\Rb$, and $(X(\be), \be\in J) $ be a family of positive integrable random variables. 
Assume that $\{X(\be)\}_{\be\in J}$ is a.s. continuous in $\be$.  Then,
 the function $\phi(\be)=\Eb[X(\be)]$ is  continuous in $\be$ if only if  the family 
 $\{X(\be)\}_{\be\in J}$ is uniformly integrable.
\end{lem}

Observe from \eqref{btU} that
\begin{equation}
\label{VUI}
|U_T(\be_x)|\leq 2\si T  M_T(\be_x)\, 1_{0\in\D}+ \frac{2\si}{1-\si} T^2  M_T(\be_x)\, 1_{0\in\D}\,\leq \frac{4}{1-\si} T^2  M_T(\be_x)\, 1_{0\in\D} . 
\end{equation}
 For $x_0\in[0,1]$,  we have:
 \begin{enumerate}
\item $\lim_{x\to x_0} T^2M_T(\be_x) 1_{0\in\D}= T^2 M_T(\be_{x_0}) 1_{0\in\D}$ a.s., 
 \item $T^2 M_T(\be_x) 1_{0\in\D} \geq 0$,
 \item  $\forall x$, $\Qb\Eb_0[T^2 \, M_T(\be_x) \, 1_{0\in\D}]=\Qb\Eb_{\be_x}(T^2 \, 1_{0\in\D})=\Eb_0(T^2 \, 1_{0\in\D})
 < + \infty$, since $\hat{\Eb}_0(T^2) < + \infty$ for $d\geq 8$.
 \end{enumerate}
It follows then from Lemma \ref{bd1} that the family $\{T^2M_T(\be_x) 1_{0\in\D}\}_{x \in [0,1]}$ is uniformly
integrable. By (\ref{VUI}), this is also true for the family $\{U_T(\be_x)\}_{x\to x_0}$ in a neighborhood of 
$x_0 \in [0,1]$.
Therefore, we obtain, 
$$\lim_{x\to x_0}\Qb\Eb_0(U_T(\be_x))=\Qb\Eb_0(U_T(\be_{x_0})) \mbox{ i.e. }\Qb\Eb_0(U_T(\be_x)) \mbox{ is continuous}.$$
Then, we get
$$\frac{\partial}{\partial t}\Qb\Eb_{0}\left[N_T(t)M_T(\be_t)\right]=\Qb\Eb_{0}\left[U_T(\be_t)\right].$$  
This finishes the proof of Lemma \ref{Dhspeed}.
\end{proof}

\subsubsection{Monotonicity of the speed}
\label{Mono} 
We remind the reader that $\tilde{Z}$ is defined as the  walk
$Z$ when it moves, and $\tilde{\D}$ denotes the cut times of $\tilde{Z}$. Since $T \ge 1$, the first term  is bounded  from below by its first item corresponding to $j=0$. 

{\allowdisplaybreaks
\begin{align}
\label{term1}
 \Qb \Eb_{\be_t}\left[\sum_{j=0}^{T-1} (\be_2-\be_1)(Y_j) \, 1_{0\in\D} \, 1_{Y_j\notin^{m}} \, 1_{Z_j=Z_{j+1}}\right]
& \geq \Qb \left[(\be_2-\be_1)(0)\right] P(0\in\D, Z_0=Z_1) \notag
\\
& = \frac{1}{d}  \Qb \left[(\be_2-\be_1)(0)\right] P(0\in\D).
\end{align}
}
The equality \eqref{term1} follows, since $\D:=\{n\in \Zb \mbox{ such that }Z_{(-\infty,n)}\cap Z_{[n,+\infty)}=\emptyset\}$ and, therefor, $\{0\in\D\}=\{Z_{-1}\neq Z_0, 0\in \tilde{\D}\}=\{\et_{-1}=0, 0\in \tilde{\D}\}.$ So we have $P(0\in\D, Z_0=Z_1)=P(0\in\tilde{\D}, \et_0=1, \et_{-1}=0)=P(0\in\tilde{\D}, \et_{-1}=0).P(\et_0=1)=\frac{1}{d}P(0\in\D).$

Now, we focus on the second term. Since $\Eb_{\be_t}[\E_k1_{Y_k\notin^m}/(1+\be_t(Y_k)\E_k)|\G_k]=0$, then
{\allowdisplaybreaks
\begin{align}
& \Qb\Eb_{\be_t}\left[\sum_{0\leq j \leq i\leq T-1}
\be_t(Y_j) \, 1_{0\in\D} \, 1_{Y_j\notin^{m}} 1_{Z_j=Z_{j+1}} 
\frac{(\be_2-\be_1)(Y_i) \E_i }{1+\be_t(Y_i)\E_i}1_{Y_i\notin^m} 1_{Z_i=Z_{i+1}} \right]
=\notag\\& \Qb\Eb_{\be_t}\left[\sum_{0\leq j \leq i\leq T-1}
\be_t(Y_j) \, 1_{0\in\D} \, 1_{Y_j\notin^{m}} 1_{Z_j=Z_{j+1}} 
(\be_2-\be_1)(Y_i)1_{Z_i=Z_{i+1}}\Eb_{\be_t}\left(\frac{\E_i1_{Y_i\notin^m} }{1+\be_t(Y_i)\E_i} |\G_i\right)\right]\notag
\\
&=0.
\end{align}
Then the second term of \eqref{dht} is equal to:
{\allowdisplaybreaks
\begin{align*}
& \Qb\Eb_{\be_t}\left[\sum_{0\leq i < j\leq T-1}
\be_t(Y_j) \, 1_{0\in\D} \, 1_{Y_j\notin^{m}} 1_{Z_j=Z_{j+1}} 
\frac{(\be_2-\be_1)(Y_i) \E_i }{1+\be_t(Y_i)\E_i}1_{Y_i\notin^m} 1_{Z_i=Z_{i+1}} \right]
\\
&\geq - \si\Qb\Eb_{\be_t}\left[\sum_{0\leq i < j\leq T-1}
\, 1_{0\in\D} \,  1_{Z_j=Z_{j+1}} 
(\be_2-\be_1)(Y_i)\Eb_{\be_t}\left(\frac{|\E_i| }{1+\be_t(Y_i)\E_i}|{\G_i}\right)1_{Y_i\notin^m} 1_{Z_i=Z_{i+1}} \right]
\\
& (\mbox{ since } |\be_t| \leq \si )\,
\\
&\geq - \si\Qb\Eb_{\be_t}\left[\sum_{0\leq i < j\leq T-1}
\, 1_{0\in\D} \,  1_{Z_j=Z_{j+1}} 
(\be_2-\be_1)(Y_i)1_{Z_i=Z_{i+1}}1_{Y_i\notin^m}1_{Z_i=Z_{i+1}} \right]
\\
&\geq - \si \Qb \Eb_{\be_t}\left[\sum_{0\leq i < j}
(\be_2-\be_1)(Y_i) 1_{Z_j=Z_{j+1}}1_{Z_i=Z_{i+1}} 1_{0\in\D} 1_{T>j}  \right] \, ,
\\
&
\geq - \si \Qb \Eb_{\be_t}\left[\sum_{0\leq i < j}
(\be_2-\be_1)(Y_i) 1_{\eta_j=1}1_{\eta_i=1} 1_{0\in\tilde{D}} 1_{\tilde{T}>\sum_{k=0}^{j-1}(1-\eta_k)}  \right]\\
&
\geq - \si \Qb \Eb_{\be_t}\left[\sum_{0\leq i < j}
(\be_2-\be_1)(Y_i) 1_{\eta_j=1}1_{\eta_i=1} 1_{0\in\tilde{D}} 1_{\tilde{T}>\sum_{k=0,k\neq i}^{j-1}(1-\eta_k)}  \right]\\
&
=  - \frac{\si}{d^2} \Qb \Eb_{\be_t}\left[\sum_{0\leq i < j}(\be_2-\be_1)(Y_i) 1_{0\in\D} 1_{\tilde{T}>\sum_{k=0,k\neq i}^{j-1}(1-\eta_k)}\right]\\
&\mbox{ since  } \eta_j\mbox{ is independent of }  \tilde{Z}, \F^Y_i, \eta_1,...,\eta_{j-1} \text{ and }\et_i\text{ is independent of }\tilde{Z}, \F^Y_i,\{\et_k\}_{k\neq i}\, , 
\\
&
=  - \frac{\si}{d^2} \Qb \Eb_{\be_t}\left[\sum_{0\leq i < j}(\be_2-\be_1)(Y_i) 1_{0\in\D} 1_{\tilde{T}>\sum_{k=0}^{j-2}(1-\eta_k)}\right]\\
&\text{with the convention that the sum over an empty set equals to }0\\
&
=  - \frac{\si}{d^2} \Qb \Eb_{\be_t}\left[\sum_{0\leq i < j}(\be_2-\be_1)(Y_i) 1_{0\in\D} 1_{T>j-1}\right]\label{EQ},\tag{EQ}\\
&= - \frac{\si}{d^2} \sum_{i=1}^{+\infty}\Qb \Eb_{\be_t}\left[(\be_2-\be_1)(Y_i) (T-i)1_{T>i}
1_{0\in\D}\right] \, , 
\\
&=  - \frac{\si}{d^2}  \sum_{i=1}^{+\infty} \Qb \Eb_{\be_t}
\left[\sum_{z\in\Zb^{d-1}} \sum_{x\in\Zb}
\frac{(\be_2-\be_1)(y)}{d^{2}}1_{Z_i=z} 1_{X_i=x} (T-i)1_{T>i}1_{0\in\D}\right] 
\mbox{ with } y=(x,z) \, , 
\\
&=  - \frac{\si}{d^2} \sum_{i\geq 1} \Qb \left[(\be_2-\be_1)(0)
\sum_{z\in\Zb^{d-1}} \sum_{x\in\Zb}
\Eb_{\th_y\be_t}(1_{Y_i=y} (T-i)1_{T>i}1_{0\in\D})\right]\\
&\text{ because } \be\text{ is stationary,}
\\
&\geq  - \frac{\si}{d^2} 
\sum_{i\geq 1} \Qb \left\{(\be_2-\be_1)(0) \sum_{z\in\Zb^{d-1}}
\Eb_{\th_y\be_t}[(2i+1)(1_{Z_i=z} (T-i)1_{T>i}1_{0\in\D})]\right\} 
\\
& \hspace{5cm} \text{ for } X_i=x \Rightarrow |x|\leq i \, ,
 \\
&\geq - \frac{\si}{d^2}  \sum_{i\geq 1}\Qb\left\{(\be_2-\be_1)(0)
\sum_{z\in\Zb^{d-1}}\Eb_{0}[(2T+1) 1_{Z_i=z} (T-i)1_{T>i}1_{0\in\D})]\right\} \, ,
\\
& \geq - \frac{\si}{d^2}  \Qb\left[(\be_2-\be_1)(0) \right]
\Eb_0\left[\frac{(2T+1)T(T+1)}{2}1_{0\in\D}\right] \, .
\end{align*}
Therefore, we get
$$ \hat{E}(T) \frac{\partial}{\partial t}f(t)\geq \frac{1}{d} 
\Qb[(\be_2-\be_1)(0)] \left[1-\frac{1}{d}\si\hat{E}\left[\frac{(2T+1)T(T+1)}{2}\right]\right].$$
This implies that $\frac{\partial}{\partial t}f(t)\geq 0$ when $d\geq\si\hat{E}\left[\frac{(2T+1)T(T+1)}{2}\right]$. Lemma \ref{TL2}} asserts that

\begin{align}\label{T3}
d_0:=\max\left\{\left\lfloor\sup_{d\geq 10}\hat{E}\left[\frac{(2T+1)T(T+1)}{2}\right]\right\rfloor+1,10\right\}<+\infty.
\end{align} 
Then, for $d\geq d_0>\si d_0$, we have $({\partial}/{\partial t})f(t)\geq 0$, which implies that $f(0)\leq f(1)$ 
 so that $v(\Qb_{\be_1})\leq v(\Qb_{\be_2}) $ on the set of probability measures on bounded
 environment. 
Choose $\si_0=\frac{10}{d_0}$, then we have the monotonicity 
for environments bounded by $\si_0$ for any $d\geq 10.$ For $d\geq d_0$, we have proved the monotonicity on the set of environments bounded by $\si<1$, take $\si$ tending to $1$, this finishes the proof. 

\section{Proof of Theorem \ref{ThERW}.}\label{Sec3}
The proof of Theorem \ref{ThERW} is based on that of Theorem \ref{Therwrc}.
\subsection{The differentiability of the speed $v(\be)$}
In the proof of Theorem \ref{Therwrc}, Section \ref{ft} about the differentiability of $f(t)$ for $d\geq 8$, we consider $m=1,\be_1(y)=0, \be_2(y)=\be_2, \be_t=t\be_2, t\in[0,1]$ for all $y\in\Zb^d$ and $\be_2$ is constant in $(0,1).$ The function $f(t)$ is difined by the couple of the environments $\be_1,\be_2$ so we denote $f_c(t)$ to be the function defined by $\be_1=0,\be_2=c$ for some constant $c\in[0,1).$ Then we have $v(\be)=f_{\be_2}(\frac{\be}{\be_2})$, moreover $f(t)$ is differentiable in $t\in [0,1]$, this implies that $v(\be)$ is differentiable in $\be\in [0,\be_2)$ for all $\be_2<1$ i.e. it is differentiable in [0,1) when $d\geq 8.$

We are now interested  in proving the existence and computing 
the derivative at the critical point $0$.
By Lemma $\ref{20}$, with $N_n:=d\sum_{j=0}^{n-1}1_{Y_j\notin}1_{Z_j=Z_{j+1}}$, we get 
$$\frac{v(\be)}{\be}=\frac{1}{d} \frac{\Eb_0(N_T1_{0\in\D}M_T(\be))}
{\Eb_0(T1_{0\in\D})}=\frac{1}{d} \Eb_0(N_T1_{0\in\D}M_T(\be)). $$
Note that 
\begin{itemize}
\item $T1_{0\in\D}M_T(\be)\geq 0$,
\item $\lim_{\be\to 0}(T1_{0\in\D}M_T(\be))=T1_{0\in\D}$,
\item $\Eb_0(T1_{0\in\D}M_T(\be))=\Eb_{\be}(T1_{0\in\D})=\Eb_0(T1_{0\in\D})
=1$ for $d \geq 6$.
\end{itemize}
Therefore, by Lemma \ref{bd1},  $\{T1_{0\in\D}M_T(\be)\}_{\be}$ is uniformly integrable in a neighborhood of $0$.
 This is also true for $\{N_T1_{0\in\D}M_T(\be)\}_{\be\to 0}$ since $N_T \leq dT$. Therefore, we get $$\lim_{\be\to 0}\Eb_0(N_T1_{0\in\D}M_T(\be))=\Eb_0(N_T1_{0\in\D}).$$
On the other hand, with $R_n$ is the range of the simple symmetric random walk on $\Zb^d$ and denote $\{Y_i\notin\}:=\{Y_i\notin\{Y_0,Y_1,...,Y_{i-1}\}\}$ then 
\begin{align*}
R(0):&=\lim_{n\to\infty}\frac{R_n}{n}=\lim_{n\to\infty}\frac{R_{T_n}}{T_n}=\lim_{n\to\infty}\frac{R_{T_1}+(R_{T_2}-R_{T_1})+...+(R_{T_n}-R_{T_{n-1}})}{T_n}\\
&=\lim_{n\to\infty}\frac{(1_{Y_0\notin}+...+1_{Y_{T_1-1}\notin})+(1_{Y_{T_1}\notin}+...+1_{Y_{T_2-1}\notin})+...+(1_{Y_{T_{n-1}}\notin}+...+1_{Y_{T_n-1}\notin})}{T_n}\\
&=\frac{\Eb_0(R_T1_{0\in\D})}{\Eb_0(T1_{0\in\D})}=\Eb_0(R_T1_{0\in\D}), (\text{because }\Eb_0(T1_{0\in\D})=\hat{\Eb}_0(T)\Pb_0(0\in\D)=1).
\end{align*}
Similarly, with $N_n=d\sum_{j=0}^{n-1}1_{Y_j\notin}1_{Z_j=Z_{j+1}}$ then 
$$
\lim_{n\to\infty}\frac{N_n}{n}=\frac{\Eb_0(N_T1_{0\in\D})}{\Eb_0(T1_{0\in\D})}=\Eb_0(N_T1_{0\in\D}).$$
Note that
$$
\Eb_0(N_n)=d\sum_{j=0}^{n-1}\Eb_0(1_{Y_j\notin}1_{Z_j=Z_{j+1}})=d\sum_{j=0}^{n-1}\Eb_0(1_{Y_j\notin})\Pb_0({Z_j=Z_{j+1}})=\Eb_0\left(\sum_{j=0}^{n-1}1_{Y_j\notin}\right)=\Eb_0(R_n).
$$ Therefore
\begin{align*}
R(0):=\lim_{n\to\infty}\frac{R_n}{n}=\lim_{n\to\infty}\Eb_0\left(\frac{R_n}{n}\right)=\lim_{n\to\infty}\Eb_0\left(\frac{N_n}{n}\right)=\Eb_0(N_T1_{0\in\D}).
\end{align*}
\subsection{Monotonicity of $v(\be)$}
In Section \ref{Mono}, we consider the particular case $m=1$ and $\be_1(y)=\be_1, \be_2(y)=\be_2$ for all $y\in\Zb^d,$ where $\be_1$ and $\be_2$ are two constants in $[0,1)$ such that $\be_1\leq \be_2\leq\si< 1$.
By \eqref{term1} and \eqref{EQ} we get that
\begin{align}\label{T2}
\hat{E}(T) \frac{\partial}{\partial t}f(t)&\geq \frac{1}{d}(\be_2-\be_1)\left[P(0\in\D)-\frac{\si}{d} \Qb \Eb_{\be_t}\left(\sum_{j\geq 1}j1_{0\in\D} 1_{T>j-1}\right)\right]\notag\\
&\geq \frac{1}{d}(\be_2-\be_1)\left[P(0\in\D)-\frac{\si}{d}E\left(\sum_{j\geq 1}j1_{0\in\D} 1_{T>j-1}\right)\right]\notag\\
&\geq \frac{1}{d}(\be_2-\be_1)\left[1-\frac{\si}{d}\hat{E}\left(\frac{T^2+T}{2}\right)\right].
\end{align}
Set $d_0:=\max\left\{\left\lfloor \sup_{d\geq 8}\hat{E}\left(\frac{T^2+T}{2}\right)\right \rfloor+1,8\right\}$ then
$\frac{\partial}{\partial t}f(t)\geq 0$ i.e. f(t) is increasing in $t\in[0,1]$ and $v(\be)$ is increasing in $\be\in[0,1]$ when $d\geq d_0$ or $\si\leq \frac{8}{d_0}$ for all $d\geq 8.$

\section{Proof of Theorem \ref{md}}\label{Sec4}
For $m-$ERW, we denote the function $f(t)$ by $f_c(m,t)$ in the case of the couple environments such that $\be_1=0,\be_2=c$ where $c$ is a constant in $[0,1)$ and $\be_t=tc, t\in [0,1].$

Set $$N_n^{m}=d\sum_{j=0}^{n-1}1_{Y_j\notin^{m}}1_{Z_j=Z_{j+1}} \, .$$
Then, from the formula \eqref{NumV}  we get
$$ \Eb_{m,\be}(X_T \, 1_{0\in\D})=\frac{\be}{d}\Eb_{m,\be}(N^{m}_T \, 1_{0\in\D})
.$$
$m-$ERW is the particular case of $m-$ERW with i.i.d. random cookies, then the law of large numbers gives the following formula of the speed when $d\geq 6$:
\begin{align}
v(m,\be)&=\frac{\Eb_{m,\be}(X_T \, 1_{0\in\D})}{\Eb_{m,\be}(T \,1_{0\in\D})}
=\frac{\be}{d}\frac{\Eb_{0}(N_T^{m}\,1_{0\in\D})}{\Eb_{0}(T\,1_{0\in\D})}.
\end{align}
We see that $v(m,\be)=f_c(m,\frac{\be}{c})$ (where $t=\frac{\be}{c}$), then 
$$
\frac{\partial v}{\partial\be}(m,\be)=\frac{\partial f_c}{\partial t}(m,\frac{\be}{c}). \frac{1}{c}\, ,
$$
and combine with the formula \eqref{dht} we obtain the derivative of the speed:
\begin{align}
\frac{\partial v}{\partial\be}(m,\be)&=\frac{1}{d}\frac{\Eb_{0}[N_T^{m}\, M_T^{m}(\be)\,1_{0\in\D}]}{\Eb_{0}(T\,1_{0\in\D})}
+\frac{\be}{d}\frac{\Eb_{0}[N_T^{m}\, M_T^{m}(\be) \, U_T^{m}(\be)\, 1_{0\in\D}]}{\Eb_{0}(T\,1_{0\in\D})}, \text{ for } \be\in [0,1)
\label{21}
\end{align}
where 
$$ U_T^{m}(\be)=\sum_{j=0}^{T-1}\frac{\E_j}{1+
\be\E_j}1_{Y_j\notin^{m}\{Y_0,...Y_{j-1}\}}1_{Z_j=Z_{j+1}}
\, ,
 $$
 $$M_T^{m}(\be)=\prod_{j=0}^{T-1}\left[1+\ep_j\be 1_{Y_j\notin^{m}\{Y_0,...Y_{j-1}\}}\right]
 \, .
 $$
 
 In order to prove the uniform convergence of $({\partial v}/{\partial\be})(m,\be)$ as $m$ goes to $+ \infty$, we use
 the following lemma, whose proof is given below:

 \begin{lem}\label{bd2}
Let $J$ be an interval of $\Rb$, and $\{X_n(\be)\}_{\be \in J, n \geq 1}$, $\{X(\be)\}_{\be\in J}$
be families of non-negative random variables. Assume that 
\begin{enumerate}
\item for every $n$, $\{ X_n(\be)\}_{\be\in J}$ is uniformly integrable,
\item $\{X(\be)\}_{\be\in J}$ is uniformly integrable,
\item $X_n(\be)$ converges in probability to $X(\be)$, uniformly in $\be$: for any $\ep >0$, 
$$ \lim_{n \to +\infty} \sup_{\be \in J} \Pb ( \left| X_n(\be) - X(\be) \right| > \ep) =0 \, . 
$$ 
\end{enumerate}
Then, $\lim_{n \rightarrow + \infty}  \sup_{\be \in J} \left| \Eb(X_n(\be))  -\Eb(X(\be)) \right| = 0$
 if and only if $\{X_n(\be)\}_{n\in\Nb,\be\in J}$ is uniformly integrable.
\end{lem}

Set  
$$N_T^{\infty}=d \sum_{j=0}^{T-1} 1_{Z_j=Z_{j+1}} \, , \, \, 
U_T^{\infty}(\be)=\sum_{j=0}^{T-1}\frac{\E_j}{1+\be\E_j}1_{Z_j=Z_{j+1}} \, , \, \, 
\,M_T^{\infty}(\be)=\prod_{j=0}^{T-1}\left(1+\ep_j\be\right)\, . $$
One can check that the following inequalities hold: 
$ \forall m \in \Nb \cup \{+ \infty \}$ , $\forall \beta \in [0, \be_0)$ ($\be_0 < 1$),
$$ N^m_T \le dT \, , \,\, M^m_T(\beta)  \leq 2^T  \, , \,\,  
V^m_T(\beta)  \leq \frac{T}{1-\be_0} \, ,
$$
$$ \left| N^m_T - N^{\infty}_T \right| \leq d (T-m)_+ \, ,
$$
$$
\sup_{\be \in [0,1]} \left| M^m_T(\be) - M^{\infty}_T(\be) \right|  \leq 2^T (T-m)_+ \, ,
$$
$$
\sup_{\be \in [0,\be_0]} \left| V^m_T(\be) - V^{\infty}_T(\be) \right|  \leq \frac{1}{1-\be_0} (T-m)_+ \, .
$$ 
We deduce from these inequalities that 
$\sup_{\be \in [0,1]} \left| N^m_ T M^m_T(\be) - N^{\infty}_T M^{\infty}_T(\be) \right|$
converges a.s. to 0 when $m$ tends to $\infty$. The same is true for 
$$\sup_{\be \in [0,\be_0]} \left| N^m_ T M^m_T(\be) V^m_T(\be)
- N^{\infty}_T M^{\infty}_T(\be) V^{\infty}_T(\be)\right|.$$


Using Lemma \ref{bd1}, we can also show that for every $m \geq 1$ the family 
$\{T M^m_T(\beta) 1_{0\in \D}\}_{\beta \in [0,1]}$
is uniformly integrable w.r.t. index $\be$ for $d \geq 6$. Indeed, it is a.s. continuous in $\beta$ for every $m\geq 1$, and for $d \geq 6$,
$$\Eb_0(T M^m_T(\beta)  \, 1_{0 \in \D}) = \Eb_{m,\be}(T\, 1_{0 \in \D})= \Eb_0(T 1_{0 \in \D}) = 1 .$$
Since $N^m_T \le dT$,  for every $m \geq 1$ the family 
$\{N^m_T M^m_T(\beta) 1_{0\in \D}\}_{\beta \in [0,1]}$
is uniformly integrable for $d \geq 6$.

In the same way, Lemma \ref{bd1} implies that for every $m \geq 1$ the family 
$\{T^2 M^m_T(\beta) $ $1_{0\in \D}\}$ $_{\beta \in [0,1]}$
is uniformly integrable for $d \geq 8$. Since $N^m_T \le T$ and $V^m_T(\be) \le \frac{1}{1-\be_0} T$ for $0 \leq \be \leq \be_0 <1$,
for every $m \geq 1$ the family $\{N^m_T V^m_T(\be)  M^m_T(\beta) 1_{0\in \D}\}_{\beta \in [0,\be_0]}$ is also uniformly integrable.  
To apply Lemma \ref{bd2}, it remains to prove that $\{N^{\infty}_T M^{\infty}_T(\beta) 1_{0\in \D}\}_{\beta \in [0,1]}$, (resp. $\{N^{\infty}_T M^{\infty}_T(\beta) V^{\infty}_T(\beta) 1_{0\in \D}\}_{\beta \in [0,1]}$) are uniformly integrable. This is
true for $d \geq 6$ (resp. $d \geq 8$) using again Lemma \ref{bd1}. 

By Lemma \ref{bd2}, we conclude that for $d \geq 8$, and $0 \leq \be_0 < 1$, 
$$ \lim_{m \rightarrow + \infty} \sup_{\be \in [0, \be_0]} \left| \frac{\partial v}{\partial\be}(m,\be) - \frac{\partial v}{\partial\be}(\infty,\be) \right| = 0 \, .
$$
Note that $\Pb_{\infty,\be}$ is the law of simple random walk with drift $\be$. Therefore,  $v(\infty,\be)=\be/d$ and $({\partial v}/{\partial\be})(\infty,\be) =1/d$, leading to the 
 statement in Theorem \ref{md}. This in turn implies that for $d\geq 8$, for all $\be_0\in[0,1)$ there exists $m(\be_0)$ such that for
  $m\geq m(\be_0)$ the speed of ERW with $m$ cookies is increasing in $\be$ on $[0,\be_0].$
  
To finish the proof of Theorem $\ref{md}$, we prove Lemma $\ref{bd2}$.

\begin{myproof}[Proof of Lemma \ref{bd2}]
$(\Leftarrow)$ We prove the sufficiency.
Since $\{X_n(\be)\}_{n,\be}$ and  $\{X(\be)\}_{\be}$ are  uniformly integrable,  for all $\ep>0,$ there exists $c_0$ 
such that for all  $c\geq c_0$, we have:
$$ \sup_{n,\be}\Eb[X_n(\be)1_{X_n(\be)\geq c}]<\ep  \, , \, \,  \sup_{\be}\Eb[X(\be)1_{X(\be)\geq c}]<\ep \, . $$
Therefore
\begin{align}\label{hoitudeu}
&|\Eb [X_n(\be)]-\Eb [X(\be)]|
\\
&\leq \ep+\Eb[|X_n(\be)|1_{|X_n(\be)-X(\be)|> \ep}]+\Eb[|X(\be)|1_{|X_n(\be)-X(\be)|> \ep}]
\notag\\
&\leq \ep+\Eb[X_n(\be)1_{X_n(\be) \geq c_0}]+\Eb[X_n(\be)1_{X_n(\be)< c_0}1_{|X_n(\be)-X(\be)|> \ep}]
\notag\\
&\quad\,\,\,\,\,+\Eb[X(\be)1_{X(\be) \geq  c_0}]+\Eb[X(\be)1_{X(\be)<c_0}1_{|X_n(\be)-X(\be)|> \ep}]\notag\\
&\leq 3 \ep+2c_0 \sup_{\be} \Pb[ |X_n(\be)-X(\be)|>\ep].
\end{align}
By assumption 3, we get that for all $\ep >0$, 
$$ \limsup_{n \rightarrow + \infty} \sup_{\be} |\Eb [X_n(\be)]-\Eb [X(\be)]| \leq 3 \ep \, .$$

($\Rightarrow$) We prove now the necessity. 
For any $C > 0$, 
\begin{align*}
& \Eb(X_n(\be) \, 1_{X_n(\be) \geq C}) 
\\ 
& = \Eb(X_n(\be) -X(\be)) + \Eb(X(\be ) \, 1_{X(\be) \geq C-1})  
+ \Eb(X(\be) \, 1_{X(\be) < C-1} - X_n(\be) \, 1_{X_n(\be) < C}) \, .
\end{align*}
Using the positivity of $X_n(\be)$, for any $\ep \in (0,1)$,
\begin{align*}
 & X(\be) \, 1_{X(\be) < C-1} - X_n(\be) \, 1_{X_n(\be) < C}
\\ 
&  \leq [X(\be)-X_n(\be)] 1_{X(\be) < C-1,X_n(\be) < C}
+ X(\be) \, 1_{X(\be) < C-1} 1_{| X_n(\be) - X(\be) | \geq \ep}
\\
& \leq \ep + 2 C 1_{| X_n(\be) - X(\be) | \geq \ep} \, . 
\end{align*}
Therefore, for any $C > 0$ and any $\ep \in (0,1)$, 
\begin{align*}
 &\sup_{\be} \Eb[X_n(\be) \, 1_{X_n(\be) \geq C}]
\\
&  \leq \sup_{\be} \left| \Eb[X_n(\be) -X(\be)] \right|
+ \sup_\be \Eb[X(\be ) \, 1_{X(\be) \geq C-1}]   + \ep + 2 C  \sup_{\be} \Pb( | X_n(\be) - X(\be) | \geq \ep) \, .
\end{align*}
Taking the limit $n \to \infty$, then $\ep \to 0$ leads to
\begin{equation}
\label{UIunif}
\limsup_{n \to \infty} \sup_{\be} \Eb[X_n(\be) \, 1_{X_n(\be) \geq C}] \leq \sup_{\be} \Eb[X(\be ) \, 1_{X(\be) \geq C-1}] 
\, .
\end{equation}
Let $\ep > 0$. Using the uniform integrability of the family $\{X(\be )\}_{\be}$, one can find $C_0(\ep)$ such
that $\sup_{\be} \Eb[X(\be ) \, 1_{X(\be) \geq C_0(\ep)-1}] \leq \ep$. By \eqref{UIunif}, there
exists $n_0(\ep)$ such that for all $n \geq n_0(\ep)$, 
$$ \sup_{\be} \Eb[X_n(\be) \, 1_{X_n(\be) \geq C_0(\ep)}] \leq 2 \ep \, .$$
For $n < n_0(\ep)$, we use the uniform integrability of the family $\{X_n(\be )\}_{\be}$ to get $C_1(\ep)$ 
such that  for any $C \geq C_1(\ep)$, $\sup_{n\leq n_0(\ep),\be}\Eb[X_n(\be)1_{X_n(\be)>C}]<\ep.$
 Now, choosing $C_2(\ep)=\max\{C_0(\ep),C_1(\ep)\}$,
we get  $\sup_{n,\be}\Eb[X_n(\be)1_{X_n(\be)>C}]<2\ep$ for all $C>C_2(\ep)$ .

\end{myproof}

\section*{Acknowledgments} I would like to thank my Ph.D. advisors  Pierre Mathieu, and Fabienne Castell for suggesting this problem. This research was supported by the French ANR project MEMEMO2 2010 BLAN 0125.
\newpage
\bibliographystyle{plain}
\bibliography{dl}
\thispagestyle{empty}
\end{document}